\documentclass[11pt,a4paper]{amsart} 
\usepackage{amsthm,amstext,amssymb,color,enumerate}
 \usepackage[all]{xy}
\usepackage{url}
\usepackage{newtxtext,newtxmath}
\usepackage{hyperref}
\usepackage[margin=2cm]{geometry}
\newcommand{\pull}{W}
\newcommand{\kk}{\mathbb{K}}
\newcommand{\R}{\mathbb{R}}

\newcommand{\Q}{\mathbb{Q}}

\newcommand{\Z}{\mathbb{Z}}

\newcommand{\Sp}{Sp}

\newcommand{\SO}{SO}
\newcommand{\SU}{SU}

\newcommand{\inprod}{\mathfrak{m}}

\numberwithin{equation}{section}
\numberwithin{figure}{section}
\theoremstyle{plain}
  \newtheorem*{thm*}{Theorem}
\newtheorem{thm}[equation]{Theorem}
  \newtheorem{prop}[equation]{Proposition}
  \newtheorem{lem}[equation]{Lemma}
  \newtheorem{cor}[equation]{Corollary}
  \newtheorem{conj}[equation]{Conjecture}
\theoremstyle{definition}
  \newtheorem{dfn}[equation]{Definition}
  \newtheorem{ex}[equation]{Example}
\theoremstyle{remark}
  \newtheorem{rem}[equation]{Remark}
  \newtheorem*{rem*}{Remark}
\theoremstyle{plain}
\makeatletter
\newtheorem*{rep@theorem}{\rep@title}
\newcommand{\newreptheorem}[2]{%
\newenvironment{rep#1}[1]{%
 \def\rep@title{#2 \ref{##1}}%
 \begin{rep@theorem}}%
 {\end{rep@theorem}}}
\newtheorem{theorem}{Theorem}

\makeatother

\newreptheorem{thm}{Theorem}
\newreptheorem{lem}{Lemma}

\begin{document}

\title{Products in equivariant homology}
\date{\today}
\author{Shizuo Kaji}
\thanks{The first named author was partially supported by KAKENHI, Grant-in-Aid for Young
     Scientists (B) 26800043 and JSPS Postdoctoral Fellowships for Research Abroad.}
\address[S. Kaji]{Department of Mathematical Sciences,
Yamaguchi University  \endgraf
1677-1, Yoshida, Yamaguchi 753-8512, Japan}
\email{skaji@yamaguchi-u.ac.jp}
\author{Haggai Tene}
\address[H. Tene]{Mathematisches Institut,
Universit\"at Heidelberg \endgraf 
Im Neuenheimer Feld 288, 69120 Heidelberg, Germany}
\email{tene@mathi.uni-heidelberg.de}


\begin{abstract}
We refine the intersection product in homology to an equivariant setting, which unifies several known constructions. 
As an application, we give a common generalisation of the Chas-Sullivan string product on a manifold 
and the Chataur-Menichi string product on the classifying space by defining a string product on the Borel construction of a manifold. 
We prove a vanishing result which enables us to define a secondary product.
The secondary product is then used to construct secondary versions of the Chataur-Menichi string product, 
and the equivariant intersection product in the Borel equivariant homology of a manifold with an action of a compact Lie group. 
The latter reduces to the product in homology of the classifying space defined by Kreck, 
which coincides with the cup product in negative Tate cohomology if the group is finite. 
\end{abstract}

\maketitle

\section{Introduction}

Let $M$ be a smooth manifold together with a smooth action of a compact 
(not necessarily connected) Lie group $G$.
Denote by $M_{G}\xrightarrow{m} BG$ its Borel construction. 
Assume the action of $G$ on $M$ is {$h$-oriented} in an appropriate sense (see Definition \ref{orientation}),
where $h$ is a multiplicative generalised (co)homology theory in the sense of \cite{Di},
 and is always assumed to satisfy the weak equivalence axiom: a weak equivalence induces an isomorphism.

Given two spaces $X$ and $Y$ over $M_G$, that is, spaces together with maps $X\to M_G$ and $Y\to M_G$,
consider the homotopy pullback diagram
\begin{equation}\label{square}
\xymatrix{
P \ar[d] \ar[r] & X \ar[d] \\
Y \ar[r] & M_G.
}
\end{equation}
 In \S \ref{sec:primary} we define a homomorphism
\[
\mu_{M_G}: h_k(X)\otimes h_l(Y) \to h_{k+l+\dim(G)-\dim(M)}(P),
\]
which we call the (primary) product for the diagram \eqref{square}.
\begin{theorem}[Theorem \ref{def-of-primary}, 
Propositions \ref{primary-property}, \ref{primary-restriction}, and \ref{primary-extension}]
The product $\mu_{M_G}$ is associative up to sign.
It enjoys the following forms of naturality:
with respect to maps of diagrams with the same base $M_G$,
with respect to closed subgroup inclusions,
and with respect to group extensions.
Dually, there exists a homomorphism in cohomology of the form $h^{m}(P)\to h^{m-dim(G)+dim(M)}(X\times Y)$.
\end{theorem}

In general, this product is non-trivial. 
Some interesting examples related to string topology and equivariant homology are given later in the introduction.
In \S \ref{sec:vanishing}, we focus on the case of the ordinary homology with coefficients in a commutative ring $R$, 
 and show that $\mu_{M_G}$ vanishes under a certain degree condition.
Denote by $\dim_R(A)$ the maximal integer $n$ such that $H_n(A;R)\neq 0$ for a space $A$, 
and denote by $F_X$ (resp. $F_Y$) the homotopy fibre of the composition $X\to M_G\xrightarrow{m} BG$
(resp. $Y\to M_G\xrightarrow{m} BG$).
\begin{theorem}[Theorem \ref{vanishing}]
The product $\mu_{M_G}: H_k(X;R) \otimes H_l(Y;R) \to H_{k+l+\dim(G)-\dim(M)}(P;R)$
is trivial if $k>\dim_R(F_X)-\dim(G)$ or $l>\dim_R(F_Y)-\dim(G)$.
\end{theorem}
In particular, this implies that the pair-of-pants product defined by Chataur-Menichi vanishes whenever $k$ or $l$ is positive
(see Example \ref{ex:Chataur-Menichi}).

We use this vanishing phenomenon to introduce a secondary version $\overline{\mu_{M_G}}$ of $\mu_{M_G}$ in \S \ref{sec:def-secondary}. 
\begin{theorem}[Theorem \ref{def-of-secondary}, Propositions \ref{secondary-naturality-wrt-equivariant-maps}, \ref{secondary-restriction} and \ref{secondary-extension}]
Let $k>\dim_R(F_X)-\dim(G)$ and $l>\dim_R(F_Y)-\dim(G)$.
There exists a homomorphism 
\[
\overline{\mu_{M_G}}:H_{k}(X;R)\otimes H_{l}(Y;R)\to H_{k+l+dim(G)-dim(M)+1}(P;R),
\]
which is natural with respect to maps of diagrams with the same base $M_G$,
with respect to closed subgroup inclusions,
and with respect to group extensions.
\end{theorem}

Now we list interesting examples of $\mu_{M_G}$. 
In the following examples, we assume all actions are suitably oriented (see Definition \ref{orientation}) without mentioning it.
\begin{enumerate}
\item Let $M=pt$. For $G$-spaces $K$ and $L$ with
classifying maps $K_G\xrightarrow{}BG, L_G\xrightarrow{}BG$,
we have the following homotopy pullback diagram:
\[
\xymatrix{
(K\times L)_G \ar[d] \ar[r] & K_G \ar[d] \\
L_G \ar[r] & BG.
}
\]
In this case, we obtain an exterior product in the homology of the Borel construction
\[
\mu_{BG}: h_k(K_G)\otimes h_l(L_G) \to h_{k+l+\dim(G)}((K\times L)_G),
\]
which reduces to the ordinary cross product when $G$ is trivial.
\item When $X=Y=M_G$ with the identity maps, we obtain an ``equivariant intersection product''
\[
\inprod: h_k(M_G)\otimes h_l(M_G) \to h_{k+l+\dim(G)-\dim(M)}(M_G),
\]
which reduces to the intersection product when $G$ is trivial.
If $h$ is the integral homology, 
equivariant homology classes can be represented by bordism classes of equivariant maps from closed, oriented stratifolds with free and orientation preserving actions of $G$ 
(see \cite{T2} and Appendix \ref{sec:stratifold}).
In this description, the product is given by transversal intersection.
\item
Let $X=Y=L(M_G)$ be the free loop space over $M_G$ with the evaluation map $ev_0: L(M_G)\to M_G$.
In this case, $P \sim\mathrm{Map}(S^1\vee S^1, M_G)$ and 
by composing $\mu_{M_G}$ with the induced map of the concatenation 
$cat: \mathrm{Map}(S^1\vee S^1, M_G) \to L(M_G)$,
we obtain a string product in the Borel construction:
\[
\psi: h_k(L(M_G))\otimes h_l(L(M_G)) \to h_{k+l+\dim(G)-\dim(M)}(L(M_G)).
\]
This is similar to the product constructed by Behrend-Ginot-Noohi-Xu (\cite{BGNX}) in the language of stacks.
This is also similar to the product by Lupercio-Uribe-Xicotencatl \cite{LMX} for orbifolds $[M/G]$ when $G$ is finite and $h$ is
 rational homology, which is later generalised by \'Angel-Backelin-Uribe \cite{Angel-Backelin-Uribe} to the integral homology.
Our product $\psi$ generalises two known constructions:
\begin{enumerate}
\item When $G$ is trivial, this reduces to the Chas-Sullivan string product for $LM$ \cite{CS,CJ}.
\item When $M=pt$, this reduces to the Chataur-Menichi pair-of-pants string product for $LBG$ \cite[Example (a)]{CM}.
\end{enumerate}
\item
More generally, let $E\to M_G$ be a fibrewise monoid.
Take $X = Y = E$ with the projection maps to $M_G$. The composition of $\mu_{M_G}$ with the fibrewise multiplication gives a product
\[
 h_k(E) \otimes h_l(E) \to h_{k+l+\dim(G)-\dim(M)}(E),
\]
which generalises Gruher-Salvatore's construction \cite{GS}.
\item
For the homotopy pullback diagram
\[
\xymatrix{
K_G \ar@{=}[d] \ar[r] & M_G \ar@{=}[d] \\
K_G \ar[r] & M_G,
}
\]
we have a pairing
\[
h_k(K_G) \otimes h_l(M_G) \to h_{k+l+\dim(G)-\dim(M)}(K_G).
\]
\end{enumerate}
\medskip

The corresponding examples of the secondary product $\overline{\mu_{M_G}}$ are given as follows.
\begin{enumerate}[(i)]
\item In the case corresponding to (1), we obtain a secondary exterior product in the Borel equivariant homology
\[
\overline{\mu_{BG}}: H^G_k(K;R)\otimes H^G_l(L;R) \to H^G_{k+l+\dim(G)+1}(K\times L;R)
\]
for $k>\dim_R(K)-\dim(G)$ and $l>\dim_R(L)-\dim(G)$.
\item In the case corresponding to (2), we obtain a secondary equivariant intersection product
\[
\overline{\inprod}: H^G_k(M;R)\otimes H^G_l(M;R) \to H^G_{k+l+\dim(G)-\dim(M)+1}(M;R)
\]
for $k,l>\dim_R(M)-\dim(G)$. 
When $M$ is a point, 
this coincides (Proposition \ref{prop-appendix}) with the product in the integral homology of $BG$ defined in an unpublished work by Kreck
 (see Appendix \ref{sec:stratifold} and \cite{K}). 
The Kreck product was described in a geometric way using stratifolds and stratifold homology.
When $G$ is finite,
the product was shown in \cite{T1} to coincide with the cup product
in negative Tate cohomology, under the identification 
$H_{k}(BG;\mathbb{Z})\xrightarrow{\cong}\widehat{H}^{-k-1}(G;\mathbb{Z})$ for $k>0$. 

Another product of the same grading 
can be defined by the tools developed in \cite{G-M} using equivariant
Poincar\'e duality and the cup product in negative
Tate cohomology under an appropriate orientability condition.
In the case when $G$ is finite and $M=pt$, it agrees with ours. 
In general we do not know if the two agree, though we expect that it is the case.
\item In the case corresponding to (3)-(b), 
composing $\overline{\mu_{BG}}$ with the induced map of the concatenation, we obtain a  product
\[
\overline{\psi}: H_{k}(LBG;R)\otimes H_{l}(LBG;R)\to H_{k+l+\dim(G)+1}(LBG;R)
\]
for $k,l>0$.
This is a secondary product to Chataur-Menichi's string product.
\item
In the case corresponding to (5), we have
\[
H^G_k(K;R) \otimes H^G_l(M;R) \to H^G_{k+l+\dim(G)-\dim(M)+1}(K;R)
\]
for $k>\dim_R(K)-\dim(G)$ and $l>\dim_R(M)-\dim(G)$.

\end{enumerate}

In \S \ref{sec:applications-secondary}, we prove some properties of the secondary product $\overline{\mu_{M_G}}$. 
In particular, we show in Theorem \ref{vanishing-general} a condition in which the secondary product vanishes.
On the other hand,
we describe several interesting examples in which $\overline{\mu_{M_G}}$ is non-trivial.
We show in Corollary \ref{ex:non-vanishing-secondary} that for every non-trivial compact Lie group $G$, 
there exists a manifold $M$ where the secondary intersection product $\overline{\mu_{M_G}}$ is non-trivial,
and we see in Proposition \ref{non-trivial-bar-psi} 
that $\overline{\psi}$ in the homology of $LBG$ is non-trivial in general.
In \S \ref{relation to the other construction}
we focus on the concrete computation of $\overline{\mu_{BG}}: H_k(BG;R)\otimes H_l(BG;R) \to H_{k+l+\dim(G)+1}(BG;R)$
 when all maps in diagram  \eqref{square} are the identity maps. In Propositions \ref{product-rank1} and \ref{computation-so3} we give a complete description
of the product in the case when $G$ is one of $S^{1},\SU(2)$ or $\SO(3)$.
We also prove the following two vanishing results:
The product in $H_*(BG;R)$ is trivial for 
\begin{itemize}
\item any compact Lie group $G$ of rank greater than one with $R=\Q$ (Proposition \ref{vanishing-rational}).
\item any connected compact classical (matrix) group of rank greater than one with $R=\Z$ (Proposition \ref{vanishing-classical}).
\end{itemize}
In the appendix, 
we show that the secondary intersection product generalises the Kreck product on $H_*(BG;\Z)$. 
\medskip

We also define a third product $\mu^{\ast}_{M_G}$ in \S \ref{sec:def-secondary}
which is related to both $\mu_{M_G}$ and $\overline{\mu_{M_G}}$. 
Consider the homotopy pushout of the upper left corner of diagram \eqref{square}
\begin{equation}\label{homotopy-join}
\xymatrix{
 P \ar[d] \ar[r] & X \ar[d]\\
 Y \ar[r] & X\bowtie_{M_G}Y.
}
\end{equation}
The space $X\bowtie_{M_G}Y$ is called the {\em homotopy join}  \cite{Do}.
\begin{theorem}[Theorem \ref{def-of-join}]
There exists a homomorphism 
\[
\mu^{\ast}_{M_G}: H_k(X;R)\otimes H_l(Y;R) \to H_{k+l+\dim(G)-\dim(M)+1}(X\bowtie_{M_G}Y;R)
\]
when $\dim(G)>0$ or $k,l>0$.
It satisfies
\[
\partial \circ \mu^{\ast}_{M_G} = \mu_{M_G}, \quad
u_* \circ \overline{\mu_{M_G}}=\mu^{\ast}_{M_G}
\]
where $\partial$ is the boundary operator in the Mayer-Vietoris sequence associated to the homotopy pushout,
and $u$ is the composition $P\to X\to X\bowtie_{M_G}Y$.
\end{theorem}
When $M=pt$ and $G$ is trivial, $\mu^{\ast}_{M_G}$ reduces to the homomorphism 
$H_k(X;R)\otimes H_l(Y;R) \to H_{k+l+1}(X*Y;R)$
defined in \cite[\S 2]{Wh}.

\medskip

We finish the introduction with a remark on the homotopy pullback diagram
\eqref{square}, which looks somewhat artificial to have $M_G$ as its base space.
In fact, the diagram \eqref{square} can be viewed in a different but equivalent way.
Given a map $X\to M_G$, consider the pullback of the universal bundle
$G\to EG\to BG$ via the composition $X\to M_G\xrightarrow{m} BG$.
Denote the total space by $K$, then we have a homotopy commutative diagram
\[
\xymatrix{
K \ar[r]^{f} \ar[d] & M \ar[r] \ar[d] & EG\ar[d] \\
K_G \ar[r]^{f_G}& M_G \ar[r] &BG
}
\]
where $f_G$ is the Borel construction of the $G$-map $f$.
Note that $X\sim K_G$ and the map $X\to M_G$ is homotopic to $f_G$.
Similarly, identify $Y\to M_G$ with the Borel construction of a $G$-map $g: L\to M$.
Taking the homotopy fibre of each space in \eqref{square} mapping down to $BG$, we obtain a homotopy pullback
\begin{equation}\label{square-upstairs}
\xymatrix{
Z \ar[d] \ar[r] & K \ar[d]^{f} \\
L \ar[r]^{g} & M,
}
\end{equation}
where $Z_G\sim P$.
Therefore, the diagram \eqref{square} can always be considered to have the form
\begin{equation}\label{equivariant-square}
\xymatrix{
Z_G \ar[d] \ar[r] & K_G \ar[d]^{f_G} \\
L_G \ar[r]^{g_G} & M_G,
}
\end{equation}
where $g_G$ and $f_G$ are the Borel constructions of some equivariant maps $f:K\to M$ and $g:L\to M$.
Conversely, the homotopy pullback \eqref{square-upstairs} consisting of equivariant maps
gives rise to the homotopy commutative diagram \eqref{equivariant-square} by taking the Borel construction.
We see that it is in fact a homotopy pullback by the five lemma.
We freely switch between the two equivalent diagrams in this paper.

\subsection*{Acknowledgements}
The authors would like to thank Matthias Kreck for helpful discussions, Peter Teichner for suggesting 
Corollary \ref{ex:non-vanishing-secondary},
and the anonymous referee for the careful reading and the precious comments.
We are grateful to the university of Bonn and the Mathematics Center Heidelberg (MATCH) for financial support during a part of this project.

\section{The construction of the product $\mu_{M_G}$}\label{sec:primary}
Throughout this paper,
we work in the category of compactly generated weak Hausdorff spaces.
Fibrations and cofibrations are in the sense of Hurewicz.

The construction of $\mu_{M_G}$ depends on umkehr (or Gysin) maps in homology of two kinds: 
one with respect to fibre bundles and the other with respect to cofibrations. 
We review these umkehr maps in the following subsections.

\subsection{The Grothendieck bundle transfer}\label{sec:grothendieck-bundle-transfer}(\cite[\S 6 (h), Chap. V]{B})
There are different ways to define umkehr maps with respect to ``nice" fibrations. 
Here we adopt Boardman's approach, called 
the {\em Grothendieck bundle transfer} \cite[\S 6 (h), Chap. V]{B} which has the advantage of allowing
arbitrary homology theories and non-connected fibres.
Note that the integration along the fibre \cite[\S 6 (e), Chap. V]{B} was used, for example, in \cite{CM}.
The integration along the fibre is defined under different conditions;
only for ordinary homology but with a larger class of fibres.
The Grothendieck bundle transfer coincides with it when both are defined. 

We briefly recall the definition of the Grothendieck bundle transfer. We refer the reader to \cite{BG,B} for details.
Let $G$ be a compact Lie group acting smoothly on a closed oriented manifold $F$.
Let $G\to E' \xrightarrow{p'} B$ be a principal $G$-bundle and 
$F\to E \xrightarrow{p} B$ be the associated $F$-bundle.
The bundle of tangents along the fibre is defined to be the vector bundle $t(p): E'\times_G TF\to E'\times_G F=E$,
where $TF$ is the tangent bundle of $F$.
There exists a fibrewise embedding $j:E \hookrightarrow U$ for some vector bundle $u: U\to B$. 
Then, by the Pontryagin-Thom construction one obtains a map between the Thom spaces 
$B^u \to E^\nu$, where $\nu$ is the normal bundle of $j$.
This, together with the fact $p^*(u)=t(p)\oplus \nu$, defines a map between the Thom spectra
$B^0 \to E^{-t(p)}$.
When an $h^*$-orientation of $-t(p)$ is specified,
it induces a homomorphism
\[
p^\natural: h_k(B) \to h_{k+\dim(F)}(E).
\]
It is easy to see the homomorphism is natural with respect to pullbacks:
for a pullback diagram
\[
\xymatrix{
\hat{E} \ar[r]^{\hat{p}} \ar[d]^{\hat{f}}  & \hat{B}\ar[d]^f \\
E \ar[r]^p & B,
}
\]
we have $\hat{f}_* \circ \hat{p}^\natural=  p^\natural\circ f_*$.

\subsection{The umkehr map with respect to a cofibration}\label{umkehr-cofibration}
(compare \cite{GS})
Let $i:A\to X$ be a cofibration. For simplicity, we assume that $A$ is a subspace of $X$. 
Fix a cohomology class $\alpha \in h^d(X,X \setminus A)$,
which we call an {\em orientation} for $i$. A typical case is when $i$ is an embedding of manifolds and $\alpha$ is an orientation class of its normal bundle.
We define an umkehr map $i^!: h_k(X) \to h_{k-d}(A)$ as follows.
For every open neighbourhood $U$  of $A$ one has the restriction homomorphism $i^!_U:h_k(X)\to  h_{k-d}(U)$: 
\[
h_k(X)\to h_k(X,X \setminus A)\cong h_k(U,U\setminus A)
\xrightarrow{\cap \alpha} h_{k-d}(U),
\]
where $\cap \alpha$ denotes the cap product with $\alpha$.
If $U'\subseteq U$ is another open neighbourhood of $A$, by naturality of the cap product we have $i^{!}_{U}=j_* \circ i^{!}_{U'}$, where $j_*$ is the map induced by the inclusion $j: U'\to U$. This implies that there exists a homomorphism $h_k(X)\to \varprojlim h_{k-d}(U)$, where the inverse limit is taken over a fundamental system of neighbourhoods. By \cite[Corollary 2]{H}, the inclusion of $A$ in its neighbourhoods induces an isomorphism $h_{*}(A)\to \varprojlim h_{*}(U)$.
By composing its inverse with the above homomorphism, we obtain the desired umkehr map $h_k(X)\to \varprojlim h_{k-d}(U) \simeq h_{k-d}(A)$.

This construction is natural in the following sense:
\begin{lem}
Suppose we are given the following topological pullback diagram
\begin{equation}\label{def-umkehr}
\xymatrix{
\hat{A} \ar[r]^{\hat{i}} \ar[d]^{\hat{f}}  & \hat{X}\ar[d]^f \\
A \ar[r]^i & X.
}
\end{equation}
If the horizontal maps are cofibrations, 
the umkehr maps with respect to a class $\alpha \in h^d(X,X \setminus A)$ and 
its pullback $f^*(\alpha)\in h^d(\hat{X},\hat{X} \setminus \hat{A})$ satisfy $\hat{f}_* \circ \hat{i}^!=i^! \circ f_*$.
\end{lem}
\begin{proof}
Since ${f}^{-1}(A)=\hat{A}$ we have a map of pairs 
$(\hat{X},\hat{X} \setminus \hat{A}) \to (X,X \setminus A)$, hence $f^*(\alpha)\in h^d(\hat{X},\hat{X} \setminus \hat{A})$. Let $U$ be an open neighbourhood of $A$ and $\hat{U}=f^{-1}(U)$. By the naturality of the cap product, we have
\[
f_* \circ \hat{i}^!_{\hat{U}}=i_U^! \circ f_*.
\]
The statement follows from the fact that this is true for every neighbourhood $U$ of $A$ and $h_{k-d}(A)\simeq \varprojlim h_{k-d}(U)$.
\end{proof}

Let $i:A\to X$ be a cofibration with an orientation $\alpha \in h^d(X,X \setminus A)$
and $f: \hat{X}\to X$ be a map.
We would like to define the umkehr map $\hat{i}^!: h_k(X) \to h_{k-d}(A)$,
where
\[
\xymatrix{
\hat{A} \ar[r]^{\hat{i}} \ar[d]^{\hat{f}}  & \hat{X}\ar[d]^{f} \\
A \ar[r]^i & X
}
\]
is a homotopy pullback.
First, we factor $f$ as the composition of a weak equivalence $u:\hat{X}\to \tilde{X}$ together with a fibration $g:\tilde{X}\to X$. Consider the following diagram:
\[
\xymatrix{
\hat{A} \ar[r]^{\hat{i}} \ar[d]^{\hat{u}}  & \hat{X}\ar[d]^{u} \\
\tilde{A} \ar[r]^{\tilde{i}} \ar[d]^{\hat{g}}  & \tilde{X}\ar[d]^{g} \\
A \ar[r]^i & X.
}
\]
where the bottom square is a topological pullback (which is also a homotopy pullback) and the map $\hat{u}$ is a whisker map, which is a weak equivalence.
Since the pullback of a cofibration along a fibration is 
again a cofibration \cite[Theorem 14.1]{Strom}, $\tilde{i}$ is a cofibration.
 Let $\tilde{i}^!$ be the umkehr map with respect to $g^{*}(\alpha)$. Then we define
\[
\hat{i}^!=\hat{u}^{-1}_{*} \circ \tilde{i}^! \circ u_{*}.
\]
By the previous Lemma, we see the umkehr map has the naturality $\hat{f}_* \circ \hat{i}^!=i^! \circ f_*$.
It is easy to see that this definition does not depend on the choice of the factorisation of $f$.



\subsection{Relation between the two umkehr maps}\label{two-umkehr-maps}
When all maps involved are smooth maps between finite closed manifolds, 
the two kinds of umkehr maps are both identified with the {\em Grothendieck transfer} \cite[\S 6 (d), Chap. V]{B}.
Let $f:N \to M$ be a smooth map between closed manifolds with an $h$-orientation of its stable normal bundle $f^*(TM)-TN$.
Choose a smooth embedding $\iota:N\to \R^n$ for some large $n$, and 
consider the embedding $f\times \iota:N \to M\times \R^n$.
We identify its normal bundle $\nu$ with the tubular neighbourhood of $N$.
Take the class $\alpha \in h^d\big ( M\times \R^n, M\times \R^n \setminus (f\times \iota)(N)\big)$ given by the inverse image of the Thom class 
by the excision isomorphism $h^*(M\times \R^n,M\times \R^n\setminus (f\times \iota)(N))\simeq h^*(\nu,\nu\setminus N)$.
Applying the construction in \S \ref{umkehr-cofibration} to the cofibration $f\times \iota$, 
we obtain a homomorphism $(f\times \iota)^!: h_k(M) \simeq h_k(M\times \R^n)\to h_{k-\dim(M)+\dim(N)}(N)$.
This construction, which is independent of the choice of $\iota$, is called the Grothendieck transfer in \cite[\S 6 (d), Chap. V]{B}. 
We denote it, by abuse of notation, by $f^!$ as well.

Assume in addition that $f$ is the projection of a fibre bundle with which the Grothendieck bundle transfer 
(\S \ref{sec:grothendieck-bundle-transfer}) is applicable.
An $h$-orientation of $f^*(TM)-TN$ amounts to an $h$-orientation of $-t(f)$ (see \cite{Dold} page 39). 
Under this correspondence $f^!$ and $f^\natural$ coincide up to sign by \cite[Theorem 6.24, Chap. V]{B}. 

\subsection{The definition of $\mu_{M_G}$}
We are now ready to define the primary product $\mu_{M_G}: h_k(X)\otimes h_l(Y) \to h_{k+l+\dim(G)-\dim(M)}(P)$
for the diagram \eqref{square}.
Observe that diagram \eqref{square}
gives rise to another homotopy pullback diagram
\[
\xymatrix{
P \ar[r] \ar[d] & M_G \ar[d]^{\Delta_{M_G}} \\
X\times Y \ar[r]^(0.42){f\times g} & M_G \times M_G,
}
\]
where $\Delta_{M_G}$ is the diagonal.
Recall that the Chas-Sullivan product is defined using 
the umkehr map with respect to the finite codimensional embedding $\Delta_M: M\to M\times M$
and the Chataur-Menichi product is defined using 
the integration along the fibre with respect to the fibration $\Delta_{BG}: BG \to BG\times BG$ whose fibre 
is the compact manifold $G$.
In contrast to these situations, in general neither $\Delta_{M_G}: M_G\to M_G\times M_G$
is finite codimensional nor its fibre $\Omega (M_G)$ is compact.
Thus, we cannot just apply the umkehr maps reviewed at the beginning of this section.
Denote the diagonal subgroup of $G\times G$ by $\Delta G$, which acts on $M\times M$.
The key observation is that one can decompose $\Delta_{M_G}$ as
\begin{equation}\label{key-decomp}
M_G \xrightarrow{\Delta_G} (M\times M)_{\Delta(G)} \xrightarrow{p_G} M_G \times M_G,
\end{equation}
where $\Delta_G$ is the equivariant diagonal and 
$p$ is obtained by further quotienting by $G\times G$.
Then, $\Delta_G$ is finite codimensional and 
$p_G$ has the fibre $(G\times G)/\Delta G$ which is homeomorphic to $G$.
Note that $p_G$ is the pullback of 
 the homogeneous bundle
 $(G\times G)/\Delta G \to (EG\times EG)_{\Delta(G)} \xrightarrow{p_G} BG \times BG$
with the structure group $G\times G$
along the direct product of the classifying maps $m\times m: M_G \times M_G \to BG \times BG$.
We identify $(EG\times EG)_{\Delta(G)}$ with $BG$,
and under this identification, the diagonal $\Delta_{BG}:BG \to BG\times BG$ is homotopic to $p_G$.
In what follows, we denote $\Delta G$ simply by $G$ when it causes no confusion.


Given diagram \eqref{square}, let $Q$ be the homotopy pullback:
\[
\xymatrix{
 Q \ar[r] \ar[d] & X \ar[d] \\
 Y \ar[r] & BG,
 }
\]
where $X\to M_G \to BG$ and $Y\to M_G \to BG$ are
compositions with the classifying map $m: M_G\to BG$.
Since $p_G$ is the homotopy pullback of $\Delta_{BG}$ via $m\times m$,
we have the following diagram with all the squares being homotopy pullbacks:
\begin{equation}\label{defining-diagram-of-mu}
\xymatrix{
P \ar[d] \ar[r]^{\hat{\Delta}_G} & Q \ar[r]^{\hat{p}_G} \ar[d] & X\times Y  \ar[d]^{} \\
M_G \ar[r]^(0.42){\Delta_G} & (M\times M)_{G}  \ar[r]^{p_G} \ar[d] & M_G \times M_G \ar[d]^{m\times m} \\
& BG \ar[r]^{p_G} & BG\times BG.
}
\end{equation}
We would like to define an $h_*(pt)$-module homomorphism $\mu_{M_G}$ by the following composition:
\[
h_k(X)\otimes h_l(Y) \xrightarrow{\times} h_{k+l}(X\times Y)\xrightarrow{\hat{p}_G^\natural} h_{k+l+\dim(G)}(Q)\xrightarrow{\hat{\Delta}_G^!}h_{k+l+\dim(G)-\dim(M)}(P).
\]
In order to define the umkehr maps  $\hat{p}_G^\natural$ and $\hat{\Delta}_G^!$
using \S \ref{sec:grothendieck-bundle-transfer} and \S \ref{umkehr-cofibration},
 the bundle of tangents along the fibre for $\hat{p}_G$ and the normal bundle of $\Delta_G$ 
should be oriented.

\begin{lem} \label{orientability}
\begin{enumerate}
\item The normal bundle of $\Delta_G$ is canonically isomorphic to the vector bundle
$(TM)_G:EG\times_G TM \to EG\times_G M=M_G$.
\item
The bundle of tangents along the fibre $t(p_G)$ for $p_G: BG\to BG\times BG$ is 
canonically isomorphic to the universal adjoint bundle $\mathrm{ad}(EG):\mathfrak{g} \to EG \times_G \mathfrak{g} \to BG$,
where $G$ acts on its Lie algebra $\mathfrak{g}$ by the adjoint action.
\end{enumerate}
\end{lem}
\begin{proof}
\begin{enumerate}
\item Observe that the bundle isomorphism given by the difference $(TM\oplus TM)/TM \to TM$
is $G$-equivariant, where $G$ acts diagonally on the source.
Thus, we identify ($G$-equivariantly) the normal bundle of $\Delta: M\to M\times M$ with $TM$.
The normal bundle of $EG\times M\to EG\times M\times M$ is then identified with
$EG \times TM \to EG\times M$. 
By taking quotient by $G$, we obtain a canonical isomorphism between the normal bundle of $\Delta_G$ and $(TM)_G$.
\item By definition, the bundle of tangents along the fibre for $p_G$ is 
\[
E(G\times G)\times_{G\times G} T((G\times G)/G) \to E(G\times G) \times_{G\times G} (G\times G)/G = BG.
\]
Observe that $(G\times G)/G \to G$ given by $(g,h) \mapsto gh^{-1}$ is 
$G\times G$-equivariant homeomorphism, where $G\times G$ acts on the target by the sandwich action.
Using the canonical isomorphism $TG\simeq G\times \mathfrak{g}$, 
this pullbacks to the $G\times G$-bundle isomorphism $T((G\times G)/G)\simeq ((G\times G)/G) \times \mathfrak{g}$,
where $\mathfrak{g}$ is acted by the sandwich action. In particular, the diagonal acts by adjunction.
Thus, we have the canonical bundle isomorphism
$E(G\times G)\times_{G\times G} T((G\times G)/G) \simeq EG \times_G \mathfrak{g}=ad(EG)$.
\end{enumerate}
\end{proof}

\begin{dfn} \label{orientation}
Suppose that $M$ is a smooth manifold together with a smooth action $\rho$ of a compact Lie group $G$. For $h$, a multiplicative generalised (co)homology theory, we say  that {\bf the triple $(M,G,\rho)$ is $h$-orientable (resp. oriented)} if the following two vector bundles, which appear in Lemma \ref{orientability}, are $h$-orientable (resp. oriented)
\begin{itemize}
\item $(TM)_G:EG\times_G TM \to EG\times_G M=M_G$, the Borel construction of the tangent bundle of $M$
\item $ad(EG): EG\times_G\mathfrak{g} \to BG$, the universal adjoint bundle of $G$.
\end{itemize}
\end{dfn}
By Lemma \ref{orientability}, we orient the bundle of tangents along the fibre $t(\hat{p}_G)$
 in \eqref{defining-diagram-of-mu}
by the pullback of the orientation of $ad(EG)$,
and the embedding $\Delta_G$ by that of $(TM)_G$.
\medskip

Note that this notion of orientability is generally stronger than the usual one:
\begin{lem}\label{triple-orientation}
An $h$-orientation of a triple induces an $h$-orientation of $M$ and $\mathfrak{g}$ such that 
the action of $G$ is orientation preserving, where $G$ acts on $\mathfrak{g}$ by adjunction.
When $h$ is the ordinary homology, the converse is also true.
\end{lem}

\begin{proof}
Observe that the bundles $(TM)_G$ and $ad(EG)$ are quotients of the equivariant bundles $EG\times TM$ and $EG\times \mathfrak{g}$
 by the free $G$-actions.
Pulling back orientations of $(TM)_G$ and $ad(EG)$, we obtain invariant orientations of $TM$ and $\mathfrak{g}$.
Conversely, in the case of ordinary homology, 
an orientation of a vector bundle is equivalent to a continuous choice of orientations of the fibres,
hence,
an invariant orientation of an equivariant vector bundle with a free action
gives rise to an orientation of its quotient.
\end{proof}

To summarise, 
\begin{thm}\label{def-of-primary}
Let $M$ be a smooth manifold with a smooth action of
a compact Lie group $G$, such that the triple $(M,G,\rho)$ is oriented. Given the homotopy pullback \eqref{square},
the umkehr maps $\hat{\Delta}_G^!$ and $\hat{p}_G^\natural$ for \eqref{defining-diagram-of-mu}
are defined and induce an $h_*(pt)$-module homomorphism 
$\mu_{M_G}: h_k(X)\otimes h_l(Y) \to h_{k+l+\dim(G)-\dim(M)}(P)$.
Dually, we have a homomorphism in cohomology of the form $h^{m}(P)\to h^{m-dim(G)+dim(M)}(X\times Y)$.
\end{thm}
\begin{rem}
{In \cite{FT}, F\'elix and Thomas define a homomorphism similar to our $\mu_{M_G}$
when the base space in \eqref{square} is a Gorenstein space.}
The Borel construction $M_G$ is a typical example of Gorenstein spaces.
They work algebraically on the level of the singular chain with coefficients in a field.
In contrast, our construction is topological and allows us to work with any generalised homology theory.
Also, our construction reveals the vanishing phenomenon Theorem \ref{vanishing} in a clear way.
(Compare with \cite[Theorem 14]{FT}.)
\end{rem}

Several examples related to equivariant homology and string topology are given in the introduction.
The following example can be derived similarly.
\begin{ex}
Let $(M,G,\rho)$ be an $h$-oriented triple.
Consider the following homotopy pullback
\[
\xymatrix{
\mathrm{Map}(S^1\vee S^1,M_G) \ar[r] \ar[d] & L(M_G) \ar[d]^{(ev_0,ev_{1/2})} \\
M_G \ar[r]^{\Delta_{M_G}} & M_G \times M_G,
}
\]
where $ev_t$ is the evaluation map at $t$. 
From the factorisation \eqref{key-decomp} of $\Delta_{M_G}$, we obtain
\[
 h_*(L(M_G)) \xrightarrow{\hat{\Delta}_G^!\circ \hat{p}_G^\natural}
 h_{*+\dim(G)-\dim(M)}(\mathrm{Map}(S^1\vee S^1,M_G)).
\]
Composing it with the map induced by the inclusion 
$\mathrm{Map}(S^1\vee S^1,M_G)\hookrightarrow L(M_G) \times L(M_G)$, one obtains the ``coproduct''
\[
h_*(L(M_G)) \to h_{*+\dim(G)-\dim(M)}(L(M_G) \times L(M_G)).
\]
This reduces to Chataur-Menichi's coproduct \cite[Example (b)]{CM} when $M=pt$ 
since $\mathrm{Map}(I,M_G)\xrightarrow{(ev_0,ev_{1})} M_G\times M_G$
is equivalent to $\Delta_{M_G}$.
\end{ex}
\begin{rem}
Recall the homotopy equivalence $LBG\sim G^{ad}_G=EG\times_G G^{ad}$,
where $G^{ad}$ is $G$ acted by itself by adjunction.
We have two different products on $h_*(LBG)\cong h_*(G^{ad}_{G})$ corresponding to (2) and (3) in \S 1:
\begin{align*}
\inprod: & h_k(G^{ad}_{G})\otimes h_l(G^{ad}_{G}) \to h_{k+l}(G^{ad}_{G}), \\
\psi: & h_k(LBG)\otimes h_l(LBG) \to h_{k+l+\dim(G)}(LBG).
\end{align*}
They are relation can be seen as follows.
Observe the following homotopy commutative diagram:
\[
\xymatrix{
 G^{ad}_G \ar[d]^\sim \ar@<0.5ex>[r]^{\Delta_G}& \ar[l]^(0.6){mul} (G^{ad}\times G^{ad})_{G} \ar[r]^{\hat{p}_G} \ar[d]^\sim 
 & G^{ad}_{G}\times G^{ad}_{G} \ar[d]^\sim \\
 LBG  & \ar[l]^(0.6){cat} \mathrm{Map}(S^1\vee S^1,BG)\ar[r]^{\hat{p}_G}  & LBG \times LBG,
}
\]
where $mul$ is induced by the multiplication of $G$ (which is equivariant with respect to the adjoint action).
By definition, $\psi = cat_* \circ \hat{p}_G^\natural = mul_* \circ \hat{p}_G^\natural$
and $\inprod = \Delta_G^! \circ \hat{p}_G^\natural$.
\end{rem}

The following properties of $\mu_{M_G}$ are easy to check:
\begin{prop}\label{primary-property}
We have
\begin{enumerate}
\item $\mu_{M_G}$ is associative up to sign.
\item $\mu_{M_G}$ commutes with homomorphism between homology theories.
\item $\mu_{M_G}$ is natural with respect to maps of diagrams with the same base $M_G$, that is,
to commutative diagrams of the form
\[
\xymatrix{
X \ar[r] \ar[d] & M_G \ar@{=}[d] & Y \ar[l] \ar[d] \\
X' \ar[r] & M_G  & Y'. \ar[l]
}
\]
\end{enumerate}
\end{prop}

Next we show that the product $\mu_{M_G}$ commutes with restriction to closed subgroups.
Consider a closed subgroup $H\subseteq G$ with its Lie algebra $\mathfrak{h}$,
and denote by $\rho|_H$ the restriction of the action $\rho$
of $G$ on $M$ to $H$.
Let $G/H \to BH \xrightarrow{i} BG$ be the associated homogeneous bundle,
and $t(i)$ be its bundle of the tangents along the fibre.
We first recall the following standard fact:
\begin{lem}\label{bundle-of-tangents}
\begin{enumerate}
\item There is a canonical $G$-equivariant isomorphism $T(G/H)\simeq G\times_H \mathfrak{g}/\mathfrak{h}$,
where $H$ acts on $\mathfrak{g}$ by adjunction.
Hence, we have $t(i)\simeq EG\times_G G\times_H \mathfrak{g}/\mathfrak{h}\simeq EG\times_H \mathfrak{g}/\mathfrak{h}$
and the following short exact sequence of vector bundles over $BH$
\begin{equation}\label{EH-EG}
0\to ad(EH) \to i^*(ad(EG)) \to  t(i)\to 0
\end{equation}
since $i^*(ad(EG)) \simeq EG\times_H \mathfrak{g}$.
\item 
Furthermore, if $K\subseteq H$ is a closed subgroup and $i': BK\to BH$ is the classifying map of the inclusion,
we have the following short exact sequence of vector bundles over $BK$
\[
 0\to EG\times_K \mathfrak{h}/\mathfrak{k} \to 
 EG\times_K \mathfrak{g}/\mathfrak{k} \to
 EG\times_K \mathfrak{g}/\mathfrak{h} \to 0,
\]
where $\mathfrak{k}$ is the Lie algebra of $K$.
In particular, we have
$t(i\circ i')\simeq (i')^*(t(i))\oplus t(i')$.
\end{enumerate}
\end{lem}
Orientations of two vector bundles in a short exact sequence
determine that of the third one
(see, for example, \cite[Chapter V, Proposition 1.10]{Rudyak}).
We always orient the bundle of tangents using the above short exact sequences
so that we have $(i\circ i')^\natural = \pm  (i')^\natural \circ i^\natural$
(\cite[Chap. V (6.1)]{B}).
\medskip

By abuse of notation, 
denote by $i$ the map $M_H\to M_G$ which is the pullback of $i$ via the classifying map $M_G\to BG$.
We can pullback the whole square \eqref{equivariant-square} via $i$ to have 
\[
\xymatrix@!0@=45pt{
 & Z_H \ar[rr]\ar'[d][dd]^{i} \ar[dl]
   & & K_H \ar[dd]^i \ar[dl]
\\
 L_H\ar[rr]\ar[dd]^i
 & & M_H \ar[dd]^(0.4)i
\\
 & Z_G \ar'[r][rr] \ar[dl]
   & & K_G \ar[dl]
\\
 L_G \ar[rr]
 & & M_G
}
\]
where we denote all the vertical maps by the same symbol $i$.
The orientations of the bundles of tangents along the fibre for $i$ are given 
by the pullback of the orientation of $BH\to BG$,
which in turn is given by those of $ad(EG)$ and $ad(EH)$ by \eqref{EH-EG}.

\begin{prop}[Restriction]\label{primary-restriction}
Suppose that $(M,G,\rho)$ and $(M,H,\rho|_H)$ are $h$-oriented
in such a way that the orientation of $(TM)_H$ is the pullback orientation of $(TM)_G$.
Then, the products $\mu_{M_G}$ and $\mu_{M_H}$ are compatible with 
the Grothendieck bundle transfer for $i$, that is, the following diagram commutes up to sign:
\[
\xymatrix{
h_k(K_G)\otimes h_l(L_G) \ar[r]^{\mu_{M_G}} \ar[d]^{i^\natural \otimes i^\natural}
 & h_{k+l+\dim(G)-\dim(M)}(Z_G) \ar[d]^{i^\natural} \\
h_{k+N}(K_H)\otimes h_{l+N}(L_H) \ar[r]^{\mu_{M_H}} & h_{k+l+2N+\dim(H)-\dim(M)}(Z_H),
}
\]
where $N=\dim(G)-\dim(H)$.
\end{prop}

\begin{proof}
We show for the diagrams below
\[
 \xymatrix{
M_H \ar[r]^(0.45){\Delta_H} \ar[d]^i & (M\times M)_{H} \ar[d]^i \\
M_G \ar[r]^(0.45){\Delta_G} & (M\times M)_{G} 
}, \quad
\xymatrix{
 (M\times M)_H \ar[r]^{p_H} \ar[d]^i & M_H\times M_H \ar[d]^{i\times i} \\
 (M\times M)_G \ar[r]^{p_G} & M_G \times M_G
 },
\]
the umkehr maps commute up to sign, that is, $\Delta^!_H\circ i^\natural = \pm i^\natural \circ \Delta^!_G$ 
and $p_H^\natural \circ (i\times i)^\natural = \pm i^\natural\circ p_G^\natural$.
Then, the similar commutativity of their pullbacks follows by the same argument.

For the left diagram, the commutativity follows from the naturality of the cap product
$i^\natural(x\cap c) = i^\natural(x)\cap i^*(c)$.
For the right diagram, observe that it is a pullback of the following:
\[
\xymatrix{
 BH \ar[r]^(0.45){p_H} \ar[d]^i & BH \times BH \ar[d]^{i\times i} \\
 BG \ar[r]^(0.45){p_G} & BG \times BG.
 }
\]
The bundle of tangents along the fibre for
$p_G\circ i=i\times i\circ p_H: BH\ \to BG\times BG$
is, by Lemma \ref{bundle-of-tangents}, 
\[
i^*(t(p_G))\oplus t(i)
\simeq p_H^*(t(i\times i)) \oplus t(p_H),
\]
where $p_H^*(t(i\times i))\simeq t(i)\oplus t(i)$,
and $t(p_G) \simeq ad(EG)$ by Lemma \ref{orientability}.
Denote an orientation of a bundle $\xi$ over $B$ by an element in the cohomology of the Thom space
$u(\xi)\in \tilde{h}^*(B^\xi)$.
Now we compare the following two orientations for $t(BH \to BG\times BG)$:
\begin{align*}
u(t(p_G\circ i))&=\pm u(i^*(ad(EG))) \wedge u(t(i)) \\
u(t(i\times i\circ p_H))&=\pm u(t(i)) \wedge u(t(i)) \wedge u(ad(EH)).
\end{align*}
Since $u(i^*(ad(EG))) = \pm u(ad(EH)) \wedge u(t(i))$, the two agree up to sign.
\end{proof}

We now see another kind of naturality of $\mu_{M_G}$ with regard to the extension of groups.
Let $1\to N \to G \xrightarrow{\gamma} G/N \to 1$
be a short exact sequence of groups.
Assume that $N$ acts on $M$ trivially so that the action $\rho$ of $G$ on $M$
induces an action $\overline{\rho}$ of $G/N$. Given a homotopy pullback
\[
\xymatrix{
P_G \ar[d] \ar[r] & X \ar[d] \\
Y \ar[r] & M_G,
}
\]
we use the induced map $\gamma: M_G \to M_{G/N}$ and the compositions $X\to M_G \xrightarrow{\gamma} M_{G/N}$
(resp. $Y\to M_G \xrightarrow{\gamma} M_{G/N}$) to form another homotopy pullback
\[
\xymatrix{
P_{G/N} \ar[d] \ar[r] & X \ar[d] \\
Y \ar[r] & M_{G/N}.
}
\]
Assuming that both triples $(M,G,\rho)$ and $(M,G/N,\bar{\rho})$ are oriented, we would like to relate  $\mu_{M_G}$ and $\mu_{M_{G/N}}$.
For this, we need an auxiliary group.
Let $K=\{ (x,y)\in G\times G \mid \gamma(x)=\gamma(y)\}$ be
the fibre product of two copies of $G$ over $G/N$.
Denote by $G \xrightarrow{d} K \xrightarrow{j} G\times G$ the inclusions of subgroups (as well as their classifying maps),
and by $k$ the homomorphism $K\to G/N$ sending $(x,y)$ to $\gamma(x)$.
\begin{lem}
The following diagram is a homotopy pullback 
\begin{equation}\label{def-of-K}
\xymatrix{
BK \ar[r]^k \ar[d]^{j} & B(G/N) \ar[d]^{p_{G/N}} \\
BG \times BG \ar[r]^(0.4){\gamma\times \gamma} & B(G/N) \times B(G/N).
}
\end{equation}
\end{lem}
\begin{proof}
Let $Z$ be the homotopy pullback of the diagram.
The whisker map $BK\to Z$ is shown to be a weak equivalence by the five lemma applied
to the bundle map from $G/N\hookrightarrow BK\to BG\times BG$ to $G/N\hookrightarrow Z\to BG\times BG$.
\end{proof}

Let $W$ be the homotopy pullback
\begin{equation}\label{K-G/N}
\xymatrix{
W \ar[r] \ar[d]^w & M_{G/N} \ar[d]^{\Delta_{G/N}} \\
(M\times M)_K \ar[r]_k & (M\times M)_{G/N}. 
}
\end{equation}

\begin{lem}
We have the following ladder of homotopy pullbacks:
\[
\xymatrix{
P_G \ar[d]^{\hat{d}} \ar[r] & M_G \ar[r]^(0.4){\Delta_G} \ar[d]^d & (M\times M)_G \ar[r] \ar[d]^d & BG \ar[d]^d \\
P_{G/N} \ar[r] & W \ar[r]^(0.4)w &  (M\times M)_K \ar[r] & BK.
}
\]
where $\hat{d}$ is the whisker map.
\end{lem}
\begin{proof}
The right-most square is a homotopy pullback by a standard fact about the homotopy quotient.
For the middle square, observe the following commutative diagram:
\[
\xymatrix{
M_G \ar[r]^(0.4){\Delta_G} \ar[d]^d \ar@{}[dr]|{\mbox{(3)}} & (M\times M)_G \ar[d] \ar[r] \ar@{}[ddr]|{\mbox{(2)}} & BG \ar[dd]^\gamma\\
W \ar[r]^(0.4)w \ar[d]  \ar@{}[dr]|{\mbox{(1)}} &  (M\times M)_K \ar[d] \\
M_{G/N} \ar[r]^(0.4){\Delta_{G/N}}   &  (M\times M)_{G/N}\ar[r] & B(G/N).
}
\]
Since the square (2) and the concatenated square (1)+(2)+(3) are pullbacks, 
so is (1)+(3) by Lemma 14 in \cite{M}.
Since the (1) is a homotopy pullback by definition, so is (3).

For the left-most square, consider the following diagram:
\begin{equation}\label{big diagram}
\xymatrix{
P_G \ar@{}[dr]|{\mbox{(5)}} \ar[r] \ar[d]^{\hat{d}} & M_G \ar[d]^d \\
P_{G/N} \ar@{}[dr]|{\mbox{(4)}} \ar[r] \ar[d] & W \ar[d]^{w} \ar[r] \ar@{}[dr]|{\mbox{(2)}} & M_{G/N} \ar[d]^{\Delta_{G/N}} \\
Q_{G/N} \ar@{}[dr]|{\mbox{(3)}} \ar[r]^u \ar[d] & (M\times M)_K \ar[r] \ar[d]^j \ar@{}[dr]|{\mbox{(1)}} & (M\times M)_{G/N} \ar[d]^{p_{G/N}} \\
 X\times Y \ar[r] & M_G\times M_G \ar[r]^(0.4){\gamma\times \gamma} & M_{G/N}\times M_{G/N}.
}
\end{equation}
We claim that all the squares are homotopy pullbacks. 
Square (1) is the pullback of \eqref{def-of-K}.
This implies that the concatenated square (1)+(2) is a homotopy pullback as well. The outer square (1)+(2)+(3)+(4) is a homotopy pullback by definition, hence it follows that also the square (3)+(4) is a homotopy pullback. This, together with the fact that the square (3)+(4)+(5) is a homotopy pullback by definition, shows that the square (5) is a homotopy pullback.
\end{proof}

Now, we are ready to prove:
\begin{prop}[Extension]\label{primary-extension}
Assume that $(TM)_G\to M_G$ is given the pullback orientation of that of $(TM)_{G/N}\to M_{G/N}$.
With an appropriate choice of an orientation of $-t(d)$ for the bundle $d: BG \to BK$,
we have $$\mu_{M_G}=\pm \hat{d}^\natural \circ \mu_{M_{G/N}}.$$
\end{prop}
\begin{proof}
Consider the following homotopy commutative diagram:
\[
\xymatrix{
P_G \ar[r]^(0.4){\hat{\Delta}_G} \ar[d]_{\hat{d}} & Q_G \ar[r]^{\hat{p}_G} \ar[d]^d & X \times Y \\
 P_{G/N} \ar[r]^(0.4){\hat{\Delta}_{G/N}} & Q_{G/N} \ar[ru]_{\hat{p}_{G/N}}.
}
\]
where $Q_G$ and $Q_{G/N}$ are as in \eqref{defining-diagram-of-mu}.
The right triangle is the pullback of the triangle
\[
\xymatrix{
 BG \ar[r]^(0.4){p_G} \ar[d]^d & BG \times BG \\
 BK \ar[ru]_{j}
}
\]
along $X\times Y\to M_G\times M_G \to BG\times BG.$
Hence, by choosing the orientation of $-t(d)$ by Lemma \ref{bundle-of-tangents} applied to
the sequence $G\xrightarrow{d} K \xrightarrow{j} G\times G$, 
we have $d^{\natural} \circ \hat{p}_{G/N}^\natural= \pm \hat{p}_G^\natural$. The left square is the pullback of the  following diagram
\[
\xymatrix{
M_G \ar[r]^(0.4){\Delta_G} \ar[d]_d & (M\times M)_G  \ar[d]^d  \\
 W \ar[r]^(0.4){w} & (M\times M)_K
}
\]
along $Q_{G/N}\xrightarrow{u} (M\times M)_K$.
Give the pullback orientation of $\Delta_{G/N}$ to $w$ by \eqref{K-G/N}. Then, $\Delta_G$ is given the pullback orientation of $w$ by assumption. 
A similar argument to the one used in Proposition \ref{primary-restriction} shows that
 ${d}^{\natural} \circ w^! = \pm {\Delta}_{G}^! \circ d^{\natural}$
and $\hat{d}^{\natural} \circ \hat{\Delta}_{G/N}^! = \pm \hat{\Delta}_{G}^! \circ d^{\natural}$.
We conclude that 
$$\hat{d}^\natural \circ \mu_{M_{G/N}}=  \hat{d}^\natural \circ \hat{\Delta}_{G/N}^! 
\circ \hat{p}_{G/N}^\natural=\pm\hat{\Delta}_{G}^! \circ d^{\natural} \circ \hat{p}_{G/N}^\natural
=\pm \hat{\Delta}_{G}^! \circ \hat{p}_G^\natural = 
\pm\mu_{M_G}.$$
\end{proof}

\section{vanishing of $\mu_{M_G}$} \label{sec:vanishing}
From now on, we restrict to the case where $h$ is the ordinary homology
with coefficients in a commutative ring  $R$. We prove that in certain cases the product $\mu_{M_G}$ vanishes.
 In particular, it implies that the Chataur-Menichi pair-of-pants product vanishes in positive degrees. In the next section we use the vanishing result of $\mu_{M_G}$
  to construct a secondary product $\overline{\mu_{M_G}}$.

Recall from Lemma \ref{triple-orientation} that the triple $(M,G,\rho)$ is $R$-orientable if and only if 
the action $\rho$ on $M$ is $R$-orientation preserving and
the adjoint action of $G$ on $\mathfrak{g}$ is $R$-orientation preserving.
We will always assume that our triples are oriented. 

\begin{lem}\label{vanishing lemma}
Let $K$ and $L$ be $G$-spaces.
Given the following pullback diagram
\[
\xymatrix{
(K\times L)_G \ar[d] \ar[r] & K_G \ar[d]^{f_G} \\
L_G \ar[r]^{g_G} & BG,
}
\]
where $g_G$ and $f_G$ are the classifying maps. 
The product $\mu_{BG}: H^G_k(K;R) \otimes H^G_l(L;R) \to H^G_{k+l+\dim(G)}(K\times L;R)$
 is trivial if $k>\dim_R(K)-\dim(G)$ or $l>\dim_R(L)-\dim(G)$.
\end{lem}
\begin{proof}
Assume $k>\dim_R(K)-\dim(G)$. Let $T$ be a CW-complex of dimension $l$ and $q:T\to L_G$ a map which induces a surjection $q_*:H_l(T;R)\to H_l(L_G;R)$ (one can always find such $T$ and $q$ by taking a skeleton of a CW-approximation). 
Construct the following homotopy pullback diagram
\[
\xymatrix{
\tilde{T}\ar[r]^-{\tilde{q}} \ar[d]&(K\times L)_G \ar[d] \ar[r] & K_G \ar[d]^{f_G} \\
T\ar[r]^{q} &L_G \ar[r]^{g_G} & BG.
}
\]
By the naturality of the Grothendieck bundle transfer, we have $\mu_{BG}=\tilde{q}_* \circ \tilde{\mu}_{BG}$, where $\tilde{\mu}_{BG}$ is the primary product defined for 
 the outer diagram which involves $T$ and $K_G$. 
By the Serre spectral sequence for the fibration $K\to \tilde{T} \to T$ we have
\[
\dim_R(\tilde{T}) \leq \dim(T) +\dim_R(K) < l+k+dim(G),
\]
since the homology of $T$ with local coefficients vanishes above its dimension. Hence, by dimensional reasons,
 $\tilde{\mu}_{BG}$, and thus,  $\mu_{BG}$ vanishes.
\end{proof}

As a corollary, we have our main vanishing result in which $M$ needs not be a point:
\begin{thm}[Vanishing of the product]\label{vanishing}
For the diagram \eqref{equivariant-square}, 
the product $\mu_{M_G}: H^G_k(K;R) \otimes H^G_l(L;R) \to H^G_{k+l+\dim(G)-\dim(M)}(Z;R)$
is trivial if $k>\dim_R(K)-\dim(G)$ or $l>\dim_R(L)-\dim(G)$.
\end{thm}
\begin{proof}
In this case $\mu_{M_G}$ is defined to be the composition  
\[
\mu_{M_G}:H^G_{k}(K)\otimes H^G_{l}(L)\xrightarrow{\mu_{BG}} H^G_{k+l+\dim(G)}((K\times L)) 
\xrightarrow{\hat{\Delta}_G^!} H^G_{k+l+\dim(G)-\dim(M)}(Z)
 \] 
and by Lemma \ref{vanishing lemma} the first map is trivial.
\end{proof}

\begin{ex}
Assume the adjoint action of $G$ on $\mathfrak{g}$ is $R$-orientation preserving
and $G$ acts on $M$ orientation preservingly.
When $K=L=M$, the equivariant intersection product
\[
 \inprod: H^G_k(M;R) \otimes H_l^G(M;R) \to H^G_{k+l+\dim(G)}(M;R)
\]
is trivial when $k>\dim_R(M)-\dim(G)$ or $l>\dim_R(M)-\dim(G)$.
\end{ex}

\begin{ex}\label{ex:Chataur-Menichi}
Assume the adjoint action of $G$ on $\mathfrak{g}$ is orientation preserving.
Consider $\psi$ in \S 1 (3) when $M=pt$ and $h$ is the ordinary homology:
\[
\psi: H_{k}(LBG;R)\otimes H_{l}(LBG;R)\to H_{k+l+\dim(G)}(LBG;R).
\]
%
By Lemma \ref{vanishing lemma}, the product $\psi$ is trivial when $k>0$ or $l>0$. 
Moreover, by the gluing property \cite[Proposition 13]{CM} and the pants decomposition of the surface $F_{g,p+q}$ 
(the connected compact oriented surface of genus $g$ with $p$ incoming boundary circles and $q$ outgoing boundary circles),
the vanishing of the pair-of-pants product implies 
that the non-equivariant version of Chataur-Menichi's stringy operation associated to $F_{g,p+q}$ with field coefficients $\kk$
\[
 H_*(LBG;\kk)^{\otimes p} \to H_*(LBG;\kk)^{\otimes q}
\]
is trivial in positive degrees unless $g=0$ and $p=1$.

\end{ex}
\begin{rem}
When $G$ is finite or connected, the vanishing of
$\psi$ follows from 
Tamanoi's vanishing theorem \cite[\S4]{Ta} by invoking
Hepworth and Lahtinen's result \cite{HL} that asserts that
the Chataur-Menichi's TQFT admits a counit.
\end{rem}

The string product $\psi$ is not always trivial for generalised homology theories, as seen in the following example.

\begin{ex} \label{product in LBG}
The free loop fibration $LBG\to BG$ has a section $i$, hence $i_*:h_*(BG) \to h_*(LBG)$ is an injection. 
By Proposition \ref{primary-property} (3), 
we have $i_*\circ \inprod=\psi\circ (i_*\otimes i_*)$,
where $\inprod$ is the equivariant intersection product for $h_*(BG)$ (\S 1 (2)).
Now we show $\inprod$ is non-trivial when $h= \Omega^{fr}$ is the framed bordism, $G=S^1$, 
and $ad(ES^1)$ is oriented in the standard way.
Let $\iota=[pt]\in \Omega^{fr}_0$ and $\eta=[S^1]\in \Omega^{fr}_1$ be the generators.
Denote by $b: pt \to BS^1$ the base point inclusion.
By assumption, the bundle $S^1\hookrightarrow BS^1\xrightarrow{p_{S^1}} BS^1\times BS^1$
is oriented in such a way that $p_{S^1}^\natural(b_*(\iota)\times b_*(\iota))=b_*(\eta)$.
 Therefore, we have
$$\inprod(b_*(\eta)\otimes b_*(\iota))=p_{S^1}^\natural(b_*(\eta) \times b_*(\iota))=
\eta p_{S^1}^\natural(b_*(\iota) \times b_*(\iota))=\eta  b_*(\eta)=\eta^2 b_*(\iota)\neq 0.$$
\end{ex}

\section{The secondary product $\overline{\mu_{M_G}}$}\label{sec:def-secondary}
The vanishing of $\mu_{M_G}$ in Theorem \ref{vanishing} suggests 
 that in the case where both inequalities $k>\dim_{R}(K)-\dim(G)$ and $l>\dim_{R}(L)-\dim(G)$ hold,
  we can define a secondary product associated to \eqref{equivariant-square}
\[
\overline{\mu_{M_G}}:H^G_{k}(K;R)\otimes H^G_{l}(L;R)\to  H^G_{k+l+\dim(G)-\dim(M)+1}(Z;R).
\]
In this section we describe the construction of  $\overline{\mu_{M_G}}$ and prove some basic properties. 

We first use some auxiliary spaces $S$ and $T$,
and later we show that the construction is independent of this choice. 
Given maps $h_1:S\to K_G$ and $h_2:T \to L_G$ where $S$ and $T$ have the homotopy type of CW-complexes of dimensions $k$ and $l$ respectively, we construct a homomorphism
\[
 \varphi: H_k(S;R) \otimes H_l(T;R) \to H_{k+l+\dim(G)+1}((K\times L)_G;R).
\]
Consider the following diagram, where all the squares are homotopy pullbacks:
\begin{equation}\label{defining-diagram}
\xymatrix{
\pull \ar[r] \ar[d] & \hat{S} \ar[r] \ar[d] & S \ar[d] \\
\hat{T}\ar[r] \ar[d] & (K \times L)_G \ar[r] \ar[d] & K_G \ar[d]^{m\circ f_G} \\
T \ar[r] & L_G \ar[r]^{m\circ g_G} & BG.}
\end{equation}
We first consider the outer diagram and obtain a homomorphism
\[
\mu_{BG}: H_{k}(S;R) \otimes H_{l}(T;R) \to H_{k+l+\dim(G)}(\pull;R).
\]
Next, we construct a homomorphism (see also Appendix \ref{sec:suspension})
\begin{equation}\label{eq:phi}
\phi_W: H_{k+l+\dim(G)}(\pull;R) \to H_{k+l+\dim(G)+1}((K\times L)_G;R).
\end{equation}
For this, form the homotopy pushout of the the upper left corner of the diagram \eqref{defining-diagram}
\[
\xymatrix{
 \pull \ar[r] \ar[d] & \hat{S} \ar[d] \ar[ddr] \\
 \hat{T} \ar[r] \ar[rrd] & B \ar@{-->}[dr]^q \\ 
 & & (K \times L)_G,
 }
\]
where $q$ is the whisker map.
As in the proof of Lemma \ref{vanishing lemma}, the assumption on $S$ and  $T$ implies that  $\dim_R(\hat{S}), \dim_R(\hat{T})< k+l+\dim(G)$. Using this, we define $\phi_W$ to be the composition $q_* \circ \partial^{-1}$,
where $\partial$ is the boundary homomorphism of the Mayer-Vietoris sequence for the homotopy pushout 
\[
 0 = H_{N+1}(\hat{S};R) \oplus H_{N+1}(\hat{T};R) \to H_{N+1}(B;R) \xrightarrow{\partial} 
 H_N(\pull;R) \to  H_{N}(\hat{S};R) \oplus H_{N}(\hat{T};R) =0,
\]
where $N=k+l+\dim(G)$. By composition, we obtain a homomorphism
\begin{equation}\label{ext-prod}
\varphi: H_{k}(S;R) \otimes H_{l}(T;R) 
\xrightarrow{\phi_W \circ \mu_{BG}}
H_{k+l+\dim(G)+1}((K\times L)_G;R).
\end{equation}

\begin{lem}\label{well-definedness}
Let $S^{(i)} \ (i=1,2)$  (resp.\ $T$) be spaces which have the homotopy type of CW-complexes of dimension $k$ (resp.\ $l$). Assume we are given maps $h_S^{(i)}:S^{(i)} \to K_G$ and $h_T:T\to L_G$ and classes $\alpha^{(i)} \in H_k(S^{(i)};R)$ such that $(h^{(1)}_S)_*(\alpha^{(1)})=(h^{(2)}_S)_*(\alpha^{(2)})$. Then for every $\beta\in H_l(T;R)$  we have $\varphi(\alpha^{(1)}\otimes\beta)=\varphi(\alpha^{(2)}\otimes\beta)$.
\end{lem}
\begin{proof}
%
For each $i=1,2$, consider the diagram (\ref{defining-diagram})
 with superscript like $W^{(i)}$.
Let $h_{S'}:S'\to K_G$ be the $(k+1)$-skeleton of a CW-approximation of $K_G$.
By cellular approximation, $h^{(i)}_1$ factors, up to homotopy, as $S^{(i)} \xrightarrow{\iota^{(i)}} S' \xrightarrow{h_{S'}} K_G$
and $\iota^{(1)}_*(\alpha^{(1)})=\iota^{(2)}_*(\alpha^{(2)})$.
Consider the diagram (\ref{defining-diagram}) with $S$ replaced by $S'$.
Take the homotopy pushout of the upper-left corner to obtain
\begin{equation}\label{pushout}
\xymatrix{
 \pull' \ar[r] \ar[d] & \hat{S'} \ar[d] \ar[ddr] \\
 \hat{T} \ar[r] \ar[rrd] & B' \ar@{-->}[dr]^{q'} \\ 
 & & (K\times L)_G,
 }
\end{equation}
where $\dim_R(\hat{S'})\le N=k+l+\dim(G)$ and $q'$ is the whisker map.
By the naturality of the Mayer-Vietoris sequence,
we have the following commutative diagram of exact sequences:
\[
\xymatrix{
0 \ar[r] & H_{N+1}(B^{(i)};R) \ar[r] \ar[d] & H_N(\pull^{(i)};R) \ar[r] \ar[d]^{d^{(i)}_*} & 0 \ar[d]\\
0 \ar[r] & H_{N+1}(B';R) \ar[r] & H_N(\pull';R) \ar[r] & H_N(\hat{S'};R)\oplus 0,
}
\]
where $d^{(i)}: \pull^{(i)} \to \pull'$ is the whisker map.
Since $\mu_{BG}$ is natural we have 
$d^{(i)}_*(\mu_{BG}(\alpha^{(i)}\otimes\beta))=
\mu_{BG}(\iota^{(i)}_*(\alpha^{(i)})\otimes\beta)$,
whose right hand side does not depend on $i$.
Since the map $H_{N+1}(B';R) \to H_N(\pull';R)$ is injective,
and each $q^{(i)}$ factors up to homotopy through $B^{(i)}\to B' \xrightarrow{q'} (K\times L)_G$,
we have that $\varphi(\alpha^{(1)}\otimes\beta)=\varphi(\alpha^{(2)}\otimes\beta)$. 
\end{proof}

To define our product
for classes ${\alpha}\in H_{k}(K_G;R)$ and ${\beta}\in H_{l}(L_G;R)$,
choose two spaces $S,T$ as above, together with maps $h_S:S\to K_G$ and $h_T:T\to L_G$ 
and classes 
$\overline{\alpha}\in H_{k}(S;R)$ and $\overline{\beta}\in H_{l}(T;R)$
such that $\alpha=(h_S)_*(\overline{\alpha})$ and $\beta=(h_T)_*(\overline{\beta})$.

Using the fact that $Q\sim (K\times L)_G$ in \eqref{equivariant-square}, we define 
\[
\overline{\mu_{M_G}}(\alpha \otimes \beta) = \hat{\Delta}_G^!\circ \varphi(\overline{\alpha}\otimes\overline{\beta}).
\]
By Lemma \ref{well-definedness} we see that this is independent of the choices made.

To summarise, 
\begin{thm}\label{def-of-secondary}
Let $M$ be a smooth $R$-oriented manifold with a smooth and orientation preserving action of $G$, 
and assume that the adjoint action of $G$ on $\mathfrak{g}$ is $R$-orientation preserving. 
For the homotopy pullback \eqref{equivariant-square},
we have a homomorphism 
\[
\overline{\mu_{M_G}}:H^G_{k}(K;R)\otimes H^G_{l}(L;R)\to  H^G_{k+l+\dim(G)-\dim(M)+1}(Z;R),
\]
when $k>\dim_{R}(K)-\dim(G)$ and $l>\dim_{R}(L)-\dim(G)$.
\end{thm}

For a wider degree range, we can define another homomorphism in a similar manner.
\begin{thm}\label{def-of-join}
Assume $\dim(G)>0$ or $k,l>0$.
Given the diagram \eqref{square},
we have a homomorphism 
\[
\mu^\ast_{M_G}:H_{k}(X;R)\otimes H_{l}(Y;R)\to  H_{k+l+\dim(G)-\dim(M)+1}(X\bowtie_{M_G}Y;R),
\]
where $X\bowtie_{M_G}Y$ is the homotopy join of $X\to M_G$ and $Y\to M_G$ as in \eqref{homotopy-join}.
It satisfies
\[
\partial \circ \mu^{\ast}_{M_G} = \mu_{M_G}, \quad
u_* \circ \overline{\mu_{M_G}}=\mu^{\ast}_{M_G}
\]
where $\partial$ is the boundary operator in the Mayer-Vietoris sequence associated to the homotopy pushout
defining $X\bowtie_{M_G}Y$,
and $u$ is the composition $P\to X\to X\bowtie_{M_G}Y$.
\end{thm}
\begin{proof}
For classes ${\alpha}\in H_{k}(X;R)$ and ${\beta}\in H_{l}(Y;R)$,
choose CW-complexes $S, T$, and classes $\overline{\alpha}, \overline{\beta}$ as in the definition of 
$\overline{\mu_{M_G}}$.
Consider the diagram \eqref{defining-diagram}.
We define a map $\varphi'$ similar to $\varphi$ by the composition
\begin{align*}
H_k(S;R)\otimes H_l(T;R)& \xrightarrow{\mu_{BG}} H_{k+l+\dim(G)}(W;R)
 \xrightarrow{{\partial'}^{-1}} H_{k+l+\dim(G)+1}(S\bowtie_{BG}T;R)\\
 &\xrightarrow{q_*} H_{k+l+\dim(G)+1}(L_G\bowtie_{BG}K_G;R),
\end{align*}
where $\partial'$ is
the boundary operator in the Mayer-Vietoris sequence associated to the homotopy pushout as in the definition of \eqref{eq:phi},
and $q_*$ is induced by the whisker map $q: S\bowtie_{BG}T\to L_G\bowtie_{BG}K_G$.
Note by the degree condition, $\partial'$ is an isomorphism.
Composing $\varphi'$ with 
$\hat{\Delta}_G^!:
H_{k+l+\dim(G)+1}(L_G\bowtie_{BG}K_G;R) \to
H_{k+l+\dim(G)-\dim(M)+1}(X\bowtie_{M_G}Y;R)$,
we define 
\[
\mu^*_{M_G}(\alpha \otimes \beta) = 
\hat{\Delta}_G^!\circ \varphi'(\overline{\alpha}\otimes\overline{\beta}).
\]
Similarly to Lemma \ref{well-definedness}, we can show that it does not depend on the choices.

Its relation to $\mu_{M_G}$ and $\overline{\mu_{M_G}}$ follows easily from the definition.
\end{proof}

\begin{ex}\label{ex:ganea}
 Let $G^{*n}$ be the $n$-fold join of $G$ and 
 $B_nG=G^{*(n+1)}/G \xrightarrow{f_n} BG$ be the $n$-th stage of Milnor's construction of the classifying space of $G$.
Consider the following diagram, where the outer square is a homotopy pullback
 and the upper-left triangle is a homotopy pushout:
\begin{equation*}
\xymatrix{
(G^{*k}\times G^{*l})/G \ar[rr] \ar[dd]  & &  B_lG \ar[dd]^{f_l} \ar[dl] \\
& B_{k+l+1}G \\ 
B_kG \ar[rr]^{f_k} \ar[ur] & & BG.
}
\end{equation*}
 Here, we identify $B_{k+l+1}G \sim B_kG \bowtie_{BG} B_lG$. 
 Denote a generator of $H_{(\dim(G)+1)n}(B_nG;\Z)\cong \Z$ by $\alpha_n$.
 Then, we have $\mu^\ast_{BG}(\alpha_k, \alpha_l) = \pm \alpha_{k+l+1}$.
\end{ex}

\begin{rem}\label{alternative-def-of-secondary}
Under a stronger assumption on degrees than Theorem \ref{def-of-secondary},
we can give an alternative definition of $\overline{\mu_{M_G}}$.
For classes ${\alpha}\in H_{k}(X;R)$ and ${\beta}\in H_{l}(Y;R)$,
choose CW-complexes $S, T$, and classes $\overline{\alpha}, \overline{\beta}$ as in the definition of 
$\overline{\mu_{M_G}}$.
Consider the following diagram (instead of \eqref{defining-diagram})
with all the squares being homotopy pullback
\[
\xymatrix{
\overline{\pull} \ar[r] \ar[d] & \bar{S} \ar[r] \ar[d] & S \ar[d] \\
\bar{T}\ar[r] \ar[d] & Z_G \ar[r] \ar[d] & K_G \ar[d]^{} \\
T \ar[r] & L_G \ar[r]^{} & M_G.}
\]
Assume that $\dim_R(F_L)-\dim(G)+\dim(M)<l$ and $\dim_R(F_K)-\dim(G)+\dim(M)<k$,
where $F_L$ (resp. $F_K$) is the homotopy fibre of $L_G\xrightarrow{g_G} M_G$
(resp. of $K_G\xrightarrow{f_G} M_G$). Since $\dim_R(\bar{S})\le \dim(S)+\dim_R(F_L)$ and
$\dim_R(\bar{T})\le \dim(T)+\dim_R(F_K)$, the following is well-defined
\[
H_{k+l+\dim(G)-\dim(M)}(\overline{\pull};R)
 \xrightarrow{\phi_{\overline{W}}}H_{k+l+\dim(G)-\dim(M)+1}(Z_G;R).
\]
We consider the map
\[
\widetilde{\mu_{M_G}}(\alpha,\beta)=
\phi_{\overline{W}} \circ \mu_{M_G}(\overline{\alpha}, \overline{\beta}).
\]
Observe that $F_L$ is also the homotopy fibre of $L\xrightarrow{g} M$.
Thus, $\dim_R(L)\le\dim_R(F_L)+ \dim(M)$.
We see that the assumption
$\dim_R(F_L)-\dim(G)+\dim(M)<l$ and $\dim_R(F_K)-\dim(G)+\dim(M)<k$
implies $\dim_R(L)-\dim(G)<l$ and $\dim_R(K)-\dim(G)<k$, which is the condition required 
for $\overline{\mu_{M_G}}$ to be defined (Theorem \ref{def-of-secondary}).
In this case, we can see that $\widetilde{\mu_{M_G}}$ and $\overline{\mu_{M_G}}$ agree up to sign 
by Lemma \ref{lem:phi-commutes-with-umkehr}.
\end{rem}

\section{Properties of the secondary product $\overline{\mu_{M_G}}$}\label{sec:applications-secondary}

In this section, we discuss properties of the secondary product $\overline{\mu_{M_G}}$.
We prove a vanishing result in some cases. 
On the other hand, we see interesting non-trivial examples.

\begin{prop}\label{secondary-naturality-wrt-equivariant-maps}
(Naturality) The product $\overline{\mu_{M_G}}$ is natural with respect to ring homomorphisms $R\to R'$ of the coefficients and with respect to equivariant maps $K\to K'$ and $L\to L'$ over $M$ in degrees where it is defined. That is, for a commutative diagram
of equivariant maps
\[
\xymatrix@!0@=45pt{
 & Z \ar[rr]\ar'[d][dd]^{d} \ar[dl]
   & & K \ar[dd]^i \ar[dl]
\\
 L\ar[rr]\ar[dd]^j
 & & M \ar@{=}[dd]
\\
 & Z' \ar'[r][rr] \ar[dl]
   & & K' \ar[dl]
\\
 L' \ar[rr]
 & & M
}
\]
we have $\overline{\mu_{M_G}}(\alpha,\beta)=d_*\circ \overline{\mu_{M_G}}(i_*(\alpha), j_*(\beta))$.

\end{prop}
\begin{proof}
The first assertion is obvious. The second one follows from Lemma \ref{well-definedness}.
\end{proof}

As we see below, 
$\overline{\mu_{M_G}}$ captures the information of the singular part of the action.
\begin{cor}\label{singular set}
Assume that $K$ is a finite dimensional (rigid) $G$-CW-complex,
and denote by $\dim(K)$ its topological (covering) dimension.
Let $n$ be either the maximal dimension of a free cell $D^n\times G$ in $K$ or $-1$ if there is no such cell.
Then, $n\leq \dim(K)-\dim(G)$ since $\dim(K)$ is equal to $\max(\dim(G/H\times D^n))$, where $G/H \times D^n$ runs over all the $G$-cells of $K$ \cite{Illman}.
Let $K_\Sigma$ be the sub-complex consisting of those points with non-trivial stabiliser.
The inclusion $K_\Sigma \hookrightarrow K$ induces 
an isomorphism 
\[
H^G_*(K_\Sigma;R)\cong H^G_*(K;R)  \quad (*>n),
\]
which is compatible with $\overline{\mu_{M_G}}$ by the naturality.
This implies that we can compute 
$\overline{\mu_{M_G}}:H^G_{k}(K;R)\otimes H^G_{l}(L;R)\to  H^G_{k+l+\dim(G)-\dim(M)+1}(Z;R)$
for $k> \dim(K)-\dim(G)$ and $l> \dim(L)-\dim(G)$
by restricting it to
$H^G_{k}(K_\Sigma;R)\otimes H^G_{l}(L_\Sigma;R)$.
\end{cor}

Similarly to the primary product, $\overline{\mu_{M_G}}$ commutes with restriction to closed subgroups.

\begin{prop}[Restriction]\label{secondary-restriction}
 Let $H\subseteq G$ be a closed subgroup,
 where the adjoint actions of $H$ and $G$ on their Lie algebras are $R$-orientation preserving.
 Denote by $i$ the classifying map of the inclusion.
 Assume that $M$ is an $R$-oriented manifold with an $R$-orientation preserving action of $G$.
Then, the secondary products $\overline{\mu_{M_G}}$ and $\overline{\mu_{M_H}}$ are compatible with 
the Grothendieck bundle transfer for $i$, that is, the following diagram commutes up to sign:
\[
\xymatrix{
H_k(K_G;R)\otimes H_l(L_G;R) \ar[r]^{\overline{\mu_{M_G}}} \ar[d]^{i^\natural \otimes i^\natural}
 & H_{k+l+\dim(G)-\dim(M)+1}(Z_G;R) \ar[d]^{i^\natural} \\
H_{k+N}(K_H;R)\otimes H_{l+N}(L_H;R) \ar[r]^{\overline{\mu_{M_H}}} & H_{k+l+2N+\dim(H)-\dim(M)+1}(Z_H;R),
}
\]
where $N=\dim(G)-\dim(H)$.
\end{prop}
\begin{proof}
The secondary product $\overline{\mu_{M_G}}$ is defined as the 
composition of the cross product, $\hat{p}_G^\natural, \phi_\pull$, and $\hat{\Delta}_G^!$. 
We already showed in the proof of Proposition \ref{primary-restriction}
that the cross product, $\hat{p}_G^\natural$, and $\hat{\Delta}_G^!$ commute with the Grothendieck bundle transfer up to sign. 
Hence, it is enough to see the same is true for $\phi_\pull$,
which is shown in Lemma \ref{lem:phi-commutes-with-umkehr}.
\end{proof}

{Just as Proposition \ref{primary-extension},
we have a naturality of $\overline{\mu_{M_G}}$ with respect to group extensions.}
\begin{prop}[Extension]\label{secondary-extension}
Let $1\to N \xrightarrow{} G \xrightarrow{\gamma} G/N \to 1$
be a short exact sequence of groups.
Assume that $N$ acts on $M$ trivially
and that $(M,G,\rho)$ and $(M,G/N,\rho)$ are oriented compatibly with respect to $\gamma$.
Under the same notation as in Proposition \ref{primary-extension}, we have
\[
 \overline{\mu_{M_G}}=\pm d^\natural \circ \overline{\mu_{M_{G/N}}}.
\]
\end{prop}

Next, we show a vanishing of $\overline{\mu_{M_G}}$ for direct products.
\begin{thm}[Vanishing of the secondary product]\label{vanishing-general}
Let $G=G_1 \times G_2$ act component-wise on
$K=K_1\times K_2, L=L_1\times L_2$ and $M=M_1\times M_2$.
Set $k=k_1+k_2$ and $l=l_1+l_2$.
When the secondary products for both
\[
\xymatrix{
(Z_1)_{G_1} \ar[r] \ar[d] & (K_1)_{G_1} \ar[d] \\
(L_1)_{G_1} \ar[r] & (M_1)_{G_1}
} \quad
\xymatrix{
(Z_2)_{G_2} \ar[r] \ar[d] & (K_2)_{G_2} \ar[d] \\
(L_2)_{G_2} \ar[r] & (M_2)_{G_2}
}
\]
are defined, the secondary product 
\[ 
\overline{\mu_{M_G}}: H^{G}_{k}(K;R) \otimes 
 H^G_{l}(L;R) \to  H^G_{k+l+\dim(G)-\dim(M)+1}({Z};R),
\]
{vanishes on the image of the cross product}
\[
\big( H^{G_1}_{k_1}(K_1;R) \otimes  H^{G_2}_{k_2}(K_2;R) \big) \otimes  
 \big(H^{G_1}_{l_1}(L_1;R)  \otimes H^{G_2}_{l_2}(L_2;R)\big)
 \xrightarrow{\times \otimes \times} 
 H^{G}_{k}(K;R) \otimes 
 H^G_{l}(L;R).
\]
\end{thm}
\begin{proof}
It is not difficult to show that
$\overline{\mu_{M_G}}(\alpha_1\times \beta_1\otimes \alpha_2\times \beta_2)
=\pm \overline{\mu_{(M_1)_{G_1}}}(\alpha_1\otimes \alpha_2)\times \mu_{(M_2)_{G_2}}(\beta_1\otimes \beta_2)$. This is trivial since $\mu_{(M_2)_{G_2}}(\beta_1\otimes \beta_2)=0$ when the secondary product is defined.
\end{proof}

Taking $M_i=K_i=L_i$ to be a point we get the following corollary, which will be used later.
\begin{cor}\label{vanishing-product}
If $k_i,l_i>-\dim_R(G_i)$ 
for $i=1,2$, then the secondary product
\[
 \overline{\mu_{BG}}:
H_{k}(BG;R) \otimes 
 H_{l}(BG;R)\to
 H_{k+l+\dim(G)+1}(BG;R)
\]
{vanishes on the image of the cross product}
\[
 \big( H_{k_1}(BG_1;R) \otimes  H_{k_2}(BG_2;R) \big) \otimes  
 \big(H_{l_1}(BG_1;R)  \otimes H_{l_2}(BG_2;R)\big)
 \xrightarrow{\times \otimes \times} 
 H_{k}(BG;R) \otimes 
 H_{l}(BG;R).
\]
\end{cor}

We immediately obtain the following:
\begin{cor}\label{cor:torus}
The secondary product $\overline{\mu_{BG}}$ in the homology of $BG$ vanishes when $G$ is the torus $T^n$ {with} $n>1$.
\end{cor}

\begin{rem}
We do not know whether $\overline{\mu_{M_G}}$ is associative up to sign, though
we believe it is.
The main difficulty is that in the definition we have to choose a CW-complex $S$ and $T$ of 
the right dimension to
represent homology classes.
However, in our construction, the image of $\overline{\mu_{M_G}}$ is represented by $B$ which
is not necessarily a CW-complex of the right dimension. This prevents from directly iterating the construction. 
\end{rem}

\section{Examples of the secondary product $\overline{\mu_{M_G}}$}\label{sec:ex-secondary}
Here we present several examples of $\overline{\mu_{M_G}}$.
Throughout this section, we assume all triples $(M,G,\rho)$ are oriented.


\begin{ex}\label{ex:secondary-point-class}
{Recall the diagram in Example \ref{ex:ganea}.}
We can use this to compute partially
\[
\overline{\inprod}: H_k(BG;\Z) \otimes H_l(BG;\Z)\to H_{k+l+\dim(G)+1}(BG;\Z),
\]
where $\overline{\inprod}$ is the secondary equivariant intersection product (\S 1 (ii)) for $M=pt$.

For generators $\alpha_k \in H_{(\dim(G)+1)k}(B_nG;\Z)$,
we have
\[
\overline{\inprod}( (f_k)_*\alpha_k \otimes (f_l)_*\alpha_l) = \pm q_*(\alpha_{k+l+1}),
\]
where $q:B_{k+l+1}G\to BG$ is the whisker map.
In particular,
\[
\overline{\inprod}(\alpha_0 \otimes \alpha_0) =\pm q_*\circ \Sigma([G]) \in H_{\dim(G)+1}(BG;\Z),
\]
where $[G]$ is the fundamental class of $G$ and $q: \Sigma G\to BG$ is adjoint to the homotopy equivalence $G\xrightarrow{\sim}\Omega BG$.
\end{ex}

\begin{lem}\label{lem:induction-for-pt}
Let $i: H\subseteq G$ be the inclusion of a closed subgroup and assume that $ad(EH)$ is oriented.
Under an appropriate orientation of $(T(G/H))_G$, 
the following diagram commutes up to sign
\[
\xymatrix{
H^G_k(G/H;\Z) \otimes H^G_l(G/H;\Z)  \ar[d]^{\cong}\ar[r]^(0.52){\overline{\inprod}_G }& H^G_{k+l+\dim(H)+1}(G/H;\Z) \ar[d]^{\cong} \\
H^H_k(pt;\Z) \otimes H^H_l(pt;\Z)\ar[r]^(0.52){\overline{\inprod}_H }& H^H_{k+l+\dim(H)+1}(pt;\Z),
}
\]
where $\overline{\inprod}_H$ (resp. $\overline{\inprod}_G$) is the secondary equivariant intersection product (\S 1 (ii))
for $pt$ (resp. $G/H$).
\end{lem}
\begin{proof}
Recall the definition of the secondary equivariant intersection product from \S \ref{sec:def-secondary}.
For classes $\alpha,\beta\in H^G_*((G/H)_G;\Z)\simeq H^H_*(pt;\Z)$,
we choose maps $S\to (G/H)_G$ and $T\to (G/H)_G$, where $S$ and $T$ are CW-complexes of the appropriate dimensions, 
and lifts $\bar{\alpha}\in H_*(S;\Z)$ and $\bar{\beta}\in H_*(T;\Z)$. We can choose these maps to be fibrations.
Consider the following diagram
\[
\xymatrix{
P \ar[r]^{\hat{\Delta}_G} \ar[d] & Q \ar[r]^{\hat{p}_G} \ar[d] & S\times T \ar[d]^f\\
 (G/H)_G \ar[r]^(0.45){\Delta_G}\ar[d]^\simeq & (G/H \times G/H)_G  \ar[r]^{p_G} & 
(G/H)_G \times (G/H)_G \ar[d]^\simeq
\\
 BH  \ar[rr]^{p_H}  & & BH \times BH,  
 }
\]
where the upper squares are topological pullbacks and the lower square is strictly commutative.
Then, $\overline{\inprod}_H(\alpha,\beta)=\phi_P\circ \hat{p}_H^\natural(\bar{\alpha},\bar{\beta})$
and
$\overline{\inprod}_G(\alpha,\beta)=\Delta_G^! \circ \phi_Q\circ \hat{p}_G^\natural(\bar{\alpha},\bar{\beta})$.
By Lemma \ref{lem:phi-commutes-with-umkehr} we have
$\Delta_G^! \circ \phi_Q = \pm \phi_P\circ \hat{\Delta}_G^!$, and hence,
we would like to see $\hat{p}_H^\natural = \pm \hat{\Delta}_G^!\circ \hat{p}_G^\natural$.

Since $(G/H)_G$ has a model which is a CW-complex with finite skeleton, we can choose $S$ and $T$ to be the skeleta of the right dimension.
so that their homologies surject on the homology of $(G/H)_G$ in the dimensions.
This way we can compute the product for all classes at the same time.

Since ${S\times T}$ has the homotopy type of a finite CW-complex, there exists a closed manifold $M$ and a map 
$g:M\to {S\times T}$ which induces a surjection in homology (the dimension of $M$ does not have to coincide with that of $S\times T$). 
To see that, note that according to J.H.C. Whitehead \cite[theorem 13]{WhJ2}, every finite CW-complex has the homotopy type of finite polyhedron.
 By embedding it in a Euclidean space and taking a closed regular neighbourhood, we see it is homotopy equivalent to a compact manifold with boundary $W$.
 Choose a collar for $W$, and let $M$ be the smooth manifold obtained by gluing two copies of $M$ along their boundaries with respect to the identity map
 (this construction is called the \emph{smooth double} of $W$, see for example, \cite[VI-5]{Kosinski}). 
 Note that the fold map $M \to W$ is a retraction, hence its composition with the homotopy equivalence $W\to {S\times T}$ is the map we are looking for.  
Furthermore, considering $(G/H)_G \times (G/H)_G$ as a Hilbert manifold, then the composition $M\xrightarrow{g} {S\times T} \xrightarrow{f} (G/H)_G \times (G/H)_G$ is homotopic to a smooth map, and since $f$ is a fibration, one can choose the map $g:M\to {S\times T}$ in such a way that the composition is smooth. 

Consider the following topological pullback of the diagram above along the map $g$:
\[
\xymatrix{
M'' \ar[r]^{\tilde{\Delta}_G} \ar[d] & M' \ar[r]^{\tilde{p}_G} \ar[d] & M \ar[d]^(0.4){g}\\
P \ar[r]^{\hat{\Delta}_G} \ar[d] & Q \ar[r]^{\hat{p}_G} \ar[d] &{S\times T} \ar[d]^{f}\\
 (G/H)_G \ar[r]^(0.45){\Delta_G} & (G/H \times G/H)_G  \ar[r]^{p_G} & 
(G/H)_G \times (G/H)_G \\
 BH  \ar[rr]^{p_H} \ar@{=}[u] & & BH \times BH.  \ar@{=}[u]
}
\]
By the naturality of both umkehr maps, and the fact that $g_*$ is surjective,
it is enough to prove that $\tilde{p}_H^\natural = \pm \tilde{\Delta}_G^!\circ \tilde{p}_G^\natural$, where $\tilde{p}_H$ is the pullback
of $p_H$ along $M\to BH\times BH$, and the orientations of $\tilde{p}_H, \tilde{\Delta}_G$, and $\tilde{p}_G$ are taken to be the pullback orientations. 
The advantage of this  is that the upper row consists, as we shall see, of smooth, closed manifolds and smooth maps. 

We restrict to the first and third rows of the diagram:
\[
\xymatrix{
M'' \ar[r]^{\tilde{\Delta}_G} \ar[d] & M' \ar[r]^{\tilde{p}_G} \ar[d]^(0.4){\widetilde{f\circ g}} & M \ar[d]^(0.4){{f}\circ g}\\
 (G/H)_G \ar[r]^(0.45){\Delta_G} & (G/H \times G/H)_G  \ar[r]^{p_G} & 
(G/H)_G \times (G/H)_G.}
\]
Since $p_G$ is a proper submersion (a smooth fibre bundle with compact fibres), it is transversal to $ f\circ g$. 
This implies that $M'$ has a natural structure of a closed smooth manifold and that the maps $\tilde{p}_G$ and $\widetilde{f\circ g}$ are smooth. 
Since transversality is a local property, the bundle structure $p_H$ and $p_G$ imply that 
$\widetilde{f\circ g}$ is transversal to $\Delta_G$.
Therefore, $M''$ has a natural structure of a closed manifold and that the map  $M'' \xrightarrow{\tilde{\Delta}_G} M'$ is smooth.

By \S \ref{two-umkehr-maps} we have $\tilde{p}_H^\natural=\pm\tilde{p}_H^!$ and $\tilde{p}_G^\natural=\pm\tilde{p}_G^!$.
We will show $\tilde{p}_H^! = \pm\tilde{\Delta}_G^!\circ \tilde{p}_G^!$
by comparing the orientations of $p_G, \Delta_G$, and $p_H$ from which the orientations of 
$\tilde{p}_H, \tilde{p}_G$ and $\tilde{\Delta}_G$ are induced.
Observe that $(T(G/H))_G$ is canonically isomorphic to both the stable normal bundle $\nu(\Delta_G)$
and the  bundle of tangents along the fibre $t(i)$ associated to $i: BH\to BG$.
Thus, we orient $(T(G/H))_G= \nu(\Delta_G) = t(i)$ by the short exact sequence \eqref{EH-EG} using the orientations of $ad(EG)$ and $ad(EH)$.
Under the identification $t(p_H)=ad(EH), t(p_G)=ad(EG)$ given by Lemma \ref{orientability},
we see the orientations for 
$-t(p_H)$ and $\nu(\Delta_G)\oplus \Delta_G^*(-t(p_G))=t(i)-i^*(ad(EG))$
are equal. This implies 
$\hat{p}_G^!=(\hat{p}_G\circ \hat{\Delta}_G )^!=\pm\hat{\Delta}_G^!\circ \hat{p}_G^!$
by the composition law for the Grothendieck transfer (see \cite[\S 6 (d), Chap. V]{B}).
Therefore, we have
\[
\hat{p}_H^\natural = \pm \hat{p}_H^! = \pm (\hat{p}_G\circ \hat{\Delta}_G )^! = 
\pm\hat{\Delta}_G^!\circ \hat{p}_G^!
=\pm\hat{\Delta}_G^!\circ \hat{p}_G^\natural.
\]
\end{proof}

More generally, we believe that the following is true:
\begin{conj}
The induction isomorphism $H^G_*(M\times_H G; R) \simeq H_*^H(M; R)$
is compatible with the secondary equivariant intersection product.
\end{conj}

Lemma \ref{lem:induction-for-pt} has an interesting consequence:
\begin{cor}\label{ex:non-vanishing-secondary}
For every non trivial compact Lie group $G$ there exists a $G$-manifold such that 
the secondary equivariant intersection product is non-trivial.
\end{cor}
\begin{proof}
Take a non trivial cyclic subgroup $H\subseteq G$ when $G$ is discrete,
otherwise take $H$ to be a torus $S^1\subset G$.
Consider the $G$-manifold $M=G/H$. 
By the previous lemma, it is enough to show that the $\overline{\inprod}$ is non trivial for a point considered as a space with the trivial $H$ action. This is proved in Proposition \ref{product-rank1}.
\end{proof}


\begin{ex}[\textbf{Secondary string product in $H_*(LBG)$}]\label{sec:LBG}
Assume the same setting as in Example \ref{ex:Chataur-Menichi}.
By composing $\overline{\mu_{BG}}$ with the induced map of the concatenation map, we obtain
\[
\overline{\psi}: H_{k}(LBG;R)\otimes H_{l}(LBG;R)\to H_{k+l+\dim(G)+1}(LBG;R),
\]
where $k,l>0$.
This can be thought of as 
a secondary product for Chataur-Menichi's 
pair-of-pants product.
\end{ex}

\begin{prop}\label{non-trivial-bar-psi}
{The product $\overline{\psi}$ in $H_*(LBG;R)$ is non-trivial} if 
$\overline{\inprod}$ in $H_*(BG;R)$ is non-trivial.
Moreover, if $G$ is abelian,
the map $H_*(G;R)\otimes H_*(BG;R) \to H_*(LBG;R)$ induced by
the equivalence $G\times BG\sim LBG$ is compatible up to sign  
with the products, where the product
in $H_*(LBG;R)$ is given by $\overline{\psi}$, 
in $H_*(BG;R)$  by $\overline{\inprod}$, and in $H_*(G;R)$ by the Pontryagin product.
\end{prop}
\begin{proof}
The free loop fibration $LBG\to BG$ has a section $i$, hence $i_*:H_*(BG;R) \to H_*(LBG;R)$ is an injection. 
By naturality, we have the following commutative diagram:
\[
\xymatrix{
H_*(BG;R)\otimes H_*(BG;R) \ar[d]^{i_*\otimes i_*} \ar[r]^{\overline{\inprod}} & H_*(BG;R)\ar[d]^{d_*} \ar[rd]^{i_*}\\
H_*(LBG;R)\otimes H_*(LBG;R) \ar[r]^{\overline{\mu_{BG}}} & H_*(\mathrm{Map}(S^1\vee S^1,BG);R) 
\ar[r]^(0.6){c_*} & H_*(LBG;R),
}
\]
where $d$ is the whisker map and $c$ is the concatenation.
Since $c_*\circ d_*=i_*$ is injective, the first assertion holds.
The second assertion follows from a straightforward diagram chasing.
\end{proof}
We give examples where the product $\overline{\inprod}$ in $H_*(BG;R)$ is non-trivial in the next section.

\section{Computation of the Kreck product} \label{relation to the other construction} 

The Kreck product is a product in the homology of $BG$ with $\Z$ or $\mathbb \Z/2\Z$ coefficients,
where the adjoint action of $G$ on its Lie algebra is orientation preserving.
 It is defined in a geometric way using stratifolds. 
 We will see in the Appendix: 
 \begin{prop}
 The secondary intersection product $\overline{\inprod}$ 
 coincides with the Kreck product in the case where the latter is defined,
  that is, when $M=K=L=pt$ and $R=\mathbb Z$ or $\mathbb Z /2\Z$. 
 \end{prop}
We denote $\bar\inprod(a\otimes b)$ by $a\ast b$ when $M=K=L=pt$.
It was shown in \cite{T1} that the when $G$ is finite and $R=\mathbb{Z}$ then the Kreck product coincides with the cup product in negative Tate cohomology. In particular, we have
\begin{cor}
For a finite group $G$, 
the product $\ast$ in $H_*(BG;\Z)$ coincides with the cup product in negative Tate cohomology.
\end{cor}

This product in negative Tate cohomology is in general non-trivial (see, for example, \cite{A-M}).
However, Benson and Carlson showed that in many cases this product vanishes:
\begin{thm*}[Benson-Carlson \cite{B-C}]
Let $\kk$ be a field of characteristic $p$. 
Suppose that the $p$-rank of $G$ is greater than one. If $H^{*}(G;\kk)$ is Cohen-Macaulay then the product vanishes.
\end{thm*}

We now use our construction to compute the Kreck product including both trivial and non-trivial cases.
Namely, we give a complete computation for finite cyclic groups, $S^1, SU(2)$, and $SO(3)$.
We also prove that the product is torsion in the case where $G$ is of positive dimension and has rank greater than one. 
Let $R$ be the coefficient ring.
We need an $R$-orientation of the universal adjoint bundle $ad(EG)$ for the product to be defined.
When $G$ is connected or finite, it admits a canonical orientation
 by that of the Lie algebra $\mathfrak{g}$.
In what follows, we fix a choice of orientation of $ad(EG)$ without mentioning it specifically.

We start with a computation of the product when $G$ is a compact connected Lie group of rank one or a finite cyclic group, where it is non-trivial. We then show that if $G$ has positive dimension and its rank is greater than 
one, then the product is torsion, hence trivial in many cases (compare with \cite[9.11]{G-M}). 
This is analogous to the above theorem by Benson and Carlson.

For $S^{1}$, $\SU(2)$ and finite cyclic groups,
 we can compute the product directly from the definition using
 Example \ref{ex:secondary-point-class}:
\begin{prop}\label{product-rank1}
Let $G$ be either $S^1, SU(2)$ or a finite cyclic group.
{Recall that $H_*(BG;\Z)$ is generated by}
$a_k\in H_{(\dim(G)+1)k}(BG;\Z)$ which are 
represented by the inclusions
$S^{(\dim(G)+1)k+\dim(G)}/G \hookrightarrow S^\infty/G \sim BG$.
For $k,k'> -\dim(G)$, we have
 \[
a_{k}*a_{k'}=\pm a_{k+k'+1}.
\]
\end{prop}
\begin{proof}
In these cases, $B_nG$ in Example \ref{ex:secondary-point-class}
are nothing but the corresponding projective spaces (or lens spaces) $S^{(\dim(G)+1)k+\dim(G)}/G$.
\end{proof}

In order to compute the product in $H_*(BSO(3);\Z)$, we use the following:
\begin{prop}
\label{pro:naturality for covering}
Let $\gamma:G'\to G$ be an $n$-covering of groups which preserves the orientation
of the universal adjoint bundles.
Then, for any $\alpha, \beta\in H_*(BG';R)$ we have
\[
B\gamma_{*}(\alpha*\beta)= n\cdot B\gamma_{*}(\alpha)*B\gamma_{*}(\beta). \qquad (*)
\]
\end{prop}

\begin{proof}
{Choose $S\xrightarrow{f}BG', T\xrightarrow{g}BG'$ to represent classes} $\alpha$ and $\beta$ as in \S \ref{sec:def-secondary}.
Consider the following diagram, where the top and the bottom squares are homotopy pullback and $d$ is the whisker map:
\[
\xymatrix@!0@=45pt{
 & P' \ar[rr]\ar'[d][dd]^{d} \ar[dl]
   & & S \ar@{=}[dd] \ar[dl]^f
\\
 T\ar[rr]^(0.7)g\ar@{=}[dd]
 & & BG' \ar[dd]_(0.7){B\gamma}
\\
 & P \ar'[r][rr] \ar[dl]
   & & S \ar[dl]
\\
 T \ar[rr]
 & & BG
}
\]
Let $\phi_{P'}$ (resp. $\phi_P$) be as in \S \ref{sec:def-secondary} for the top square
(resp. for the bottom square).
Notice that we can use the same $S$ and $T$ 
to compute the products in the left and the right hand sides of (*).
Since $\phi_{P}\circ d_*=B\gamma_*\circ \phi_{P'}$, we have only to show 
$d_*\circ \mu_{BG} = n \mu_{BG'}$.
By Proposition \ref{primary-extension},
we have $\mu_{BG'}=d^\natural\circ \mu_{BG}$.
By \cite[Chap. V (6.2) (d)]{B}, 
\[
d_*\circ \mu_{BG'}=d_*\circ d^\natural\circ \mu_{BG}=
d_\natural(1) \cdot \mu_{BG},
\]
where $d_\natural$ is the Grothendieck bundle transfer on the cohomology.
Finally, by the commutative diagram of fibrations
\[
\xymatrix{
  G' \ar[r]^{\gamma} \ar[d] &  G \ar[d]^\iota \\
 P' \ar[r]^{d} \ar[d]^{\hat{p}'_G} &  P \ar[d]^{\hat{p}_G} \\
S\times T \ar@{=}[r] & S\times T.
}
\]
we see $\iota^*\circ d_\natural(1)=\gamma_\natural(1)$ by naturality.
Since $\gamma$ is $n$-covering and orientation preserving, we have $\gamma_\natural(1)=n$.
\end{proof}

Now, we give a complete description of the product in $H_*(B\SO(3);\Z)$.
We will see in the proof of Proposition \ref{vanishing-classical},
 that the product vanishes on $2$-torsion elements.
 For the free part,
let $\gamma:\SU(2)\to \SO(3)$ be the double cover. 
By a Serre spectral sequence argument
we have $\gamma_{*}(a_{k})=4^{k}b_{k}$, where $a_k\in H_{4k}(B\SU(2);\Z)$ and $b_k\in H_{4k}(B\SO(3);\Z)$
are the generators for the free parts.
By Propositions  \ref{product-rank1} and \ref{pro:naturality for covering}, we conclude:
\begin{prop}\label{computation-so3}
The only non-trivial products in $H_{*}(BSO(3);\Z)$ are
 \[
b_{k}*b_{k'}=\pm 2b_{k+k'+1}\quad(k,k'\ge0).
\]
\end{prop}

In the case where $G$ has positive dimension and its rank is greater than one, we use our vanishing theorem (Corollary \ref{vanishing-product}) to show that the product is torsion, and in many cases trivial.

First, we prove a small lemma:
\begin{lem}\label{injective-transfer}
Let $\kk$ be a field.
Let $H\subseteq G$ be a closed subgroup with specified orientations for $ad(EH)$ and $ad(EG)$.
Let $G/H \xrightarrow{j} BH \xrightarrow{i} BG$ be a homogeneous bundle whose bundle of
tangents along the fibre $t(i)$ is oriented
by \eqref{EH-EG}.
If $j^*:H^*(BH;\kk) \to H^*(G/H;\kk)$ is surjective, then the Grothendieck bundle transfer 
$i^\natural:H_*(BG;\kk)\to H_{*+\dim(G/H)}(BH;\kk)$ is injective.
\end{lem}
\begin{proof}
Let $x\neq 0 \in H_{*}(BG;\kk)$, then there exists $x^*\in H^{*}(BG;\kk)$ with 
the Kronecker pairing $\langle x^*, x \rangle \neq 0$.
Since $j^*$ is surjective, we can choose an element
 $c\in H^{\dim(G/H)}(BH;\kk)$ such that $j^*(c)$ is Poincar\'e dual to the the fundamental class of $G/H$.
Then, by \cite[Chap. V (6.2)]{B} we have $i_\natural(c)=1$, where $i_\natural$ is the Grothendieck bundle transfer on cohomology,
and 
\[
 \langle c \cup i^*(x^*), i^\natural(x) \rangle =  \pm \langle i_\natural(c \cup i^*(x^*)), x \rangle 
 = \pm \langle i_\natural(c)\cup x^*, x \rangle = \pm \langle x^*, x \rangle \neq 0,
\]
which shows $i^\natural(x)\neq 0$.
\end{proof}

\begin{prop}\label{vanishing-rational}
Let $G$ be a compact Lie group of rank greater than one.
Then, the product of two classes in $H_{*}(BG;\Z)$ is a torsion element.
\end{prop}
\begin{proof}
Since the product is natural with respect to coefficients (Proposition \ref{secondary-naturality-wrt-equivariant-maps}),
it is sufficient to show that it vanishes with the rational coefficients.
Let $T \subseteq G$ be the maximal torus of the identity component and $i: BT\to BG$ the classifying map of the inclusion.
Since $H^*(BT;\Q)\to H^*(G/T;\Q)$ is surjective, by Lemma \ref{injective-transfer} we see that
$i^\natural$ is injective with the rational coefficients. 
By Corollary \ref{cor:torus} the product vanishes in $H_{*}(BT;\Q)$,
and so does in $H_*(BG;\Q)$ by the commutativity with $i^\natural$ (Proposition \ref{secondary-restriction}).  
\end{proof}

\begin{prop}\label{vanishing-classical}
The product in $H_{*}(BG;\Z)$ vanishes for $T^{n}$, $U(n)$ and $\Sp(n)$
when $n>1$, for $\SU(n)$ when $n>2$, and for $\SO(n)$ when $n>3$.
\end{prop}
\begin{proof}
For $G=T^n, U(n), \SU(n)$ and $\Sp(n)$ the homology of $H_*(BG;\Z)$ is torsion free. Therefore,
in these cases the assertion follows from the previous proposition.
For $G=\SO(n)$,  let $D_{n-1} = (\Z/2\Z)^{n-1}$ be its maximal two torus and $i: BD_{n-1} \to B\SO(n)$
be the classifying map of the inclusion.
By naturality with respect to coefficients  (Proposition \ref{secondary-naturality-wrt-equivariant-maps}) and restriction (Proposition \ref{secondary-restriction}) the following diagram commutes up to sign
\[
\xymatrix{
H_k(B\SO(n);\Z)\otimes H_l(B\SO(n);\Z) \ar[r]^(0.62){\bar{\inprod}} \ar[d] & H_{k+l+\dim(\SO(n))+1}(B\SO(n);\Z) \ar[d] \\
H_{k+\dim(\SO(n))}(BD_{n-1};\Z/2\Z) \otimes H_{l+\dim(\SO(n))}(BD_{n-1};\Z/2\Z) \ar[r]^(0.62){\bar{\inprod}} & H_{k+l+2\dim(\SO(n))+1}(BD_{n-1};\Z/2\Z),
}
\]
where the vertical arrows are compositions of $i^\natural$ and mod-$2$ reduction.
By the previous proposition, the image of the upper row lies 
in the torsion, which injects into $H_*(BD_{n-1};\Z/2\Z)$ by the universal coefficients theorem and Lemma \ref{injective-transfer}.
By Corollary \ref{vanishing-product} the product in the lower row vanishes. This implies that the product in the upper row vanishes.
\end{proof}

\begin{rem*}
The same argument works when all the torsion elements in $H^*(BG;\Z)$ are $p$-toral, that is, 
there is a $p$-torus $D$ of rank greater than one such that $H^*(BG;\Z/p\Z)$ injects into $H^*(BD;\Z/p\Z)$ for all the torsion primes $p$
of $G$.
For example, the product in $H_*(BG_2;\Z)$ is trivial for the compact simply-connected exceptional Lie group $G_2$.
\end{rem*}

%
%
%
%

\appendix
\section{A generalisation of the Kreck product}\label{sec:stratifold}
Throughout this Appendix, 
$G$ is assumed to be a compact Lie group whose adjoint action on its Lie algebra is $R$-orientation preserving.
The Kreck product is a product in $H_*(BG;R)$ for $R=\mathbb Z$ and $\mathbb Z /2\Z$. 
Here, we describe a generalisation of the Kreck product in $H_*(BG;\Z)$ to $H^G_*(M;\Z)$,
which allows us to understand the secondary equivariant intersection product $\overline{\inprod}$ in a geometric way.
The case when the coefficients are in $\mathbb Z /2\Z$ is obtained in 
an analogous way if one replaces oriented stratifolds by stratifolds.

First, we briefly recall stratifolds and stratifold homology, and refer to
\cite{K} for a complete treatment.
Stratifolds are a generalisation of smooth manifolds. 
A stratifold is a pair $(S,F)$, where $S$ is a topological space and $F$ is a sub-sheaf of the sheaf of continuous real functions (considered as a sheaf of $\mathbb{R}$-algebras) which satisfies certain axioms. 
In particular, for every $r\in \Z_{\ge 0}$, the $r$-stratum is defined to be the set of points $x$ in $S$, 
where the tangent space at $x$ (which is defined as the space of derivations of germs at $x$) has dimension $r$.
This stratum together with the restriction of the sheaf $F$ is required to be a
smooth manifold. 
The dimension of the stratifold is defined to be the maximal integer $k$ such that
the $k$-stratum is non-empty.
A $k$-dimensional stratifold is called oriented if its $k$-th stratum
is oriented and its $k-1$ stratum is empty. In a similar way stratifolds with boundary are defined. An example of a stratifold with boundary is the cone over a stratifold $S$, denoted by $CS$. Note that if $S$ is oriented and of positive dimension then $CS$ is oriented as well. This fact implies that the corresponding bordism theory, denoted by $SH$, is naturally isomorphic to integral homology when restricted to spaces having the homotopy type of a CW-complex.

We use here a certain class of stratifolds, called $p$-stratifolds (see \cite[p. 24 example 10]{K}). Those are constructed inductively. 
We start with a discrete set $S_0$ with the sheaf of all real functions. 
The $r^{th}$ skeleton $S_r$ is defined by attaching a smooth manifold $W_r$ of dimension $r$ to $S_{r-1}$
along a smooth map 
from its boundary $\partial W_r$ and extend the sheaf in an appropriate way. This process is similar to  the construction of a CW-complex. The reason we use $p$-stratifolds is that they have the homotopy types of CW-complexes of the same dimension as the stratifolds themselves. This fact is needed in the comparison of our product with the Kreck product in Proposition \ref{prop-appendix}. From now on, when we say stratifold we mean $p$-stratifold. 
Let $M$ be a smooth manifold. 
A smooth map $g: N\to M$ between manifolds and a map from a stratifold $f: S\to M$
 are said to be {\em transversal}
when the compositions $W_r \to S_{r} \to M$ and $\partial W_r \to S_{r-1} \to M$ 
are transversal to $g$ for all $r$.
Two maps from stratifolds $f: S\to M, g: S'\to M$ are said to be transversal when 
$S\times S'\to M\times M$ is transversal to the diagonal $\Delta: M\to M\times M$.
For an oriented $r$-dimensional stratifold $S$, 
the image of a fundamental class of $(W_r, \partial W_r)$ by the composition
\[
H_r(W_r, \partial W_r;\Z) \to H_r(S_r,S_{r-2};\Z)\cong H_r(S_r;\Z)
\]
is called a fundamental class of $S$.
We need the following property of stratifolds:
\begin{lem}\label{lem:stratifold-inclusion}
A compact finite dimensional stratifold is an ANR.
In particular, a closed inclusion of stratifolds is a cofibration.
\end{lem}
\begin{proof}
A compact finite dimensional stratifold is 
a locally contractible, metrizable space,
and hence, an ANR by a theorem of Borsuk \cite{Bor}.
By a theorem of Whitehead \cite{WhJ}, 
a closed inclusion of compact ANR's is a cofibration.
\end{proof}

Given a space $X$ together
with an action of a compact Lie group $G$, the Borel $G$-equivariant homology of $X$ with the integral coefficients 
is identified with a bordism theory \cite{T2}:
Denote by $SH_{k}^{G}(X)$ the set of bordism classes of equivariant maps $f:S\to X$
from compact oriented stratifolds $S$ of dimension $k$ with a smooth, free and
orientation preserving $G$ action modulo $G$-bordism, i.e. a bordism
with a smooth, free and orientation preserving $G$ action extending the given action on the boundary. 
Since we assumed that the adjoint representation of $G$ preserves the orientation of its Lie algebra, 
we can give a canonical orientation to $S/G$.
The map $SH_{k}^{G}(X)\to SH_{k-\dim(G)}(X\times_{G}EG)\cong H_{k-\dim(G)}(X\times_{G}EG;\mathbb{Z})$ given by 
$[f:S\to X]\mapsto (f_G)_*[S/G] \in H_{k-\dim(G)}(X\times_{G}EG;\Z)$ is a natural isomorphism. 

\begin{rem*}
If the adjoint representation of $G$ does not preserve the orientation of the Lie algebra,
then $S/G$ might be non orientable, for example, $O(2)$ acts on $SO(3)$ with orbit space $\mathbb{R} P^2$. 
In this case one cannot identify $SH_{k}^{G}(X)$ with $H_{k-\dim(G)}^G(X;\mathbb{Z})$.
\end{rem*}

When $X$ is a point $pt$, 
we simply write a class $[S\to pt]\in SH^G_k(pt)\cong H_{k-\dim(G)}(BG;\Z)$ by $[S]$.
The Cartesian product
$[S\times T]$ is trivial if
at least one of $\dim(S),\dim(T)$ is positive (say $S$), since then $S\times T$
is the boundary of $CS\times T$, which has a free and orientation
preserving $G$ action (if $\dim(S)=0$ then $CS$ has a non empty codimension $1$ stratum). The Kreck product $[S]*[T]$ in $H_*(BG;\Z)$ is a secondary product defined when both $\dim(S)$ and $\dim(T)$ are positive. It is given
by gluing together the two null bordisms $CS\times T$ and $S\times CT$
along their common boundary $S\times T$, which is isomorphic to the join $S*T$; that is,
\[
[S]*[T]=[S*T],
\]
The grading is given by
\[
H_k(BG;\Z)\otimes H_l(BG;\Z) \to H_{k+l+\dim(G)+1}(BG;\Z),
\]
where $k,l>-\dim(G)$

We generalise this construction as follows. Let $X=M$ be a smooth oriented manifold with an orientation preserving action of a compact Lie group $G$.
Let $S\to M$ and $T\to M$ be maps from stratifolds $S$ and $T$.
As is proved in \cite{T2}, their transversal intersection product $S\pitchfork T\to M$
is well-defined up to bordism.
We have the following commutative diagram:
\begin{equation*}\label{transversal-intersection}
\xymatrix{
(S\pitchfork T)/G \ar[d] \ar[r] & S/G\ar[d]\\
T/G  \ar[r] & M_G.
}
\end{equation*}
Consider the following two diagrams:
\[
\xymatrix{
(S\times T)/G \ar[d] \ar[r]^{\hat{p}_G} & S/G \times T/G\ar[d] \\
BG \ar[r]^{p_G} &  BG \times BG,
}
\qquad
\xymatrix{
 (S\pitchfork T)/G \ar[r]^i \ar[d] & (S\times T)/G \ar[d]  \\
 M_G \ar[r]^{\Delta_G} & (M\times M)_G.
}
\]
The diagram on the left is a homotopy pullback diagram where the vertical maps are the classifying maps,
and the horizontal maps are fibre bundles with fibre $G$.
The diagram on the right is a topological pullback
and $i$ is a cofibration by Lemma \ref{lem:stratifold-inclusion}. 
\begin{lem}\label{lem:fundamental-class}
Let $\hat{p}^{\natural}$ be the Grothendieck bundle transfer with respect to $p$.
Then, $\hat{p}^{\natural}([S/G]\times[T/G])$ is a fundamental class of $(S\times T)/G$. 
Let $i^!$ be the umkehr map with respect to the pullback of the Thom class of $\Delta_G$.
Then, the image of a fundamental class of $(S\times T)/G$ is a fundamental class of
$(S\pitchfork T)/G$.
\end{lem}
\begin{proof}
The first assertion follows from \cite[Lemma 6.17, Chap. V]{B}.
For the second, set $Q=(S\times T)/G$ and $Q'=(S\pitchfork T)/G$.
We can relativise the construction of the umkehr map in \S \ref{sec:primary}
 to have $i^!: h_k(X,B) \to h_{k-d}(A,A\cap B)$ for a cofibered subspace $B\subseteq X$
and a Thom class in $h^d(X,X\setminus A)$. By naturality the following diagram commutes:
\[
\xymatrix{
h_k(X) \ar[d]^{i^!} \ar[r]  & h_k(X,B) \ar[d]^{i^!}\\
h_{k-d}(A)  \ar[r]  & h_{k-d}(A,B\cap A),
}
\]
Set $k=\dim(Q)$, $X=Q,A=Q'$ and $B=Q_{k-1}$ (then $B\cap A=Q'_{k-d-1}$).
Then, the horizontal maps are isomorphisms when $h$ is integral homology. Therefore, it is enough to prove that the relative fundamental class in $H_k(Q,Q_{k-1})$ is mapped to a relative fundamental class in 
$H_{k-d}(Q',Q'_{k-d-1})$. Let $W$ be the manifold with boundary used to attach the top cell in $Q$.
Then, $Q'$ is obtained by attaching $W'=W\pitchfork M_G$ to $Q'_{k-d-1}$.
By naturality, it is enough to show that the relative fundamental class
for $(W,\partial W)$ is mapped to one for $(W',\partial W')$.
This follows from the fact that the pullback of a Thom class along a transversal intersection is a Thom class.
\end{proof}
We specify the fundamental classes as
$[(S\times T)/G]=\hat{p}_G^{\natural}([S/G]\times[T/G])$ and 
$[(S\pitchfork T)/G]=i^\natural([(S\times T)/G])$.

\medskip
When $k,l>\dim(M)-\dim(G)$, we have $\dim(S/G),\dim(T/G)<\dim((S\pitchfork T)/G)=k+l+\dim(G)-\dim(M)$ so 
the map $\phi_{(S\pitchfork T)/G}:H_{\dim((S\pitchfork T)/G)}((S\pitchfork T)/G;\Z)\to H_{\dim((S\pitchfork T)/G)+1}(M_G;\Z)$
 in \S \ref{sec:def-secondary} is defined.
Suppose we have two classes $\alpha\in H_{k}^{G}(M;\Z)\cong SH_{k+\dim(G)}^{G}(M)$ and $\beta\in H_{l}^{G}(M;\Z)\cong SH_{l+\dim(G)}^{G}(M)$, represented by the maps $S\to M$ and $T\to M$.
 We define a product
\[
*:H_k^G(M;\Z)\otimes H_l^G(M;\Z) \to H^G_{k+l+\dim(G)-\dim(M)+1}(M;\Z)
\]
by $\alpha*\beta=\phi_{(S\pitchfork T)/G}([(S\pitchfork T)/G])$.
When $M=pt$, we have $S\pitchfork T=S\times T$ and the homotopy pushout of $S \leftarrow S\times T \to T$ is the join $S*T$. Therefore, this coincides with the Kreck product.

By the universal property, there is a map $d:(S\pitchfork T)/G \to P$, where $P$ is the homotopy pullback of the same diagram. We use it to give an more ``geometric" interpretation of the products $\mu_{M_G}$ and $\overline{\inprod}$.
\begin{prop}\label{prop-appendix}
\begin{enumerate}
\item The transversal intersection coincides with $\mu_{M_G}$. That is,
\[
\mu_{M_G}([S/G],[T/G])=  d_*[(S \pitchfork T)/G]\in H_{k+l+\dim(G)-\dim(M)}(P;\Z).
\]
\item 
When $k,l>\dim(M)-\dim(G)$ we have
\[
\overline{\inprod}(\alpha,\beta)=\pm \alpha*\beta,
\]
where $\overline{\inprod}$ is the secondary intersection product  (\S 1 (ii)).
\end{enumerate}
\end{prop}

\begin{proof}
\begin{enumerate}
\item
Recall from \S \ref{sec:primary} and Lemma \ref{lem:fundamental-class} that
\[
\mu_{M_G}([S/G],[T/G])=\hat{\Delta}_G^! \circ \hat{p}_G^\natural([S/G]\times [T/G])=
\hat{\Delta}_G^! ([(S\times T)/G]).
\]
In order to compute $\hat{\Delta}_G^!$ we factor the map $(S\times T)/G \to (M\times M)_G$ as the composition  
\[
(S\times T)/G \xrightarrow{j} \overline{(S\times T)/G} \xrightarrow{\pi} (M\times M)_G,
\]
where $j$ is a weak equivalence and $\pi$  a fibration. This way the following is a pullback diagram:
\[
\xymatrix{
(S \pitchfork T)/G \ar[r]^(0.55)d \ar[d]^{i} & P \ar[d]^{\hat{\Delta}_G}  \ar[r] & M_G \ar[d]^{\Delta_G}\\
(S\times T)/G \ar[r]^{\sim} & \overline{(S\times T)/G} \ar[r]^\pi & (M\times M)_G.\\
}\]
By the naturality of the umkehr map with respect to cofibrations
and Lemma \ref{lem:fundamental-class}, we conclude 
\[
\hat{\Delta}_G^! ([(S\times T)/G])=d_*\circ i^! ([(S\times T)/G])=d_*[(S \pitchfork T)/G].
\]

\item
Recall from \S \ref{sec:def-secondary} that 
\[
\overline{\inprod}(\alpha,\beta)=
{\Delta}_G^!  \circ \phi_{(S\times T)/G} \circ \hat{p}_G^\natural([S/G]\times [T/G]).
\]
Consider the following diagram, where the top and the bottom squares are homotopy pullback:
\[
\xymatrix@!0@=45pt{
 & P \ar[rr]\ar'[d][dd]^(0.4){\hat{\Delta}_G} \ar[dl]
   & & S/G \ar[dd] \ar[dl]
\\
 T/G\ar[rr]\ar[dd]
 & & M_G \ar[dd]^(0.3){\Delta_G}
\\
 & (S\times T)/G \ar'[r][rr] \ar[dl]
   & & (S\times M)/G \ar[dl]
\\
 (M\times T)/G \ar[rr]
 & & (M\times M)_G
}
\]
and denote by $\phi_P$ 
the map defined as in \S \ref{sec:def-secondary} for the top square.
In a similar way as in Proposition \ref{secondary-restriction}, we have
\[
{\Delta}_G^!\circ \phi_{(S\times T)/G} = \pm \phi_P\circ \hat{\Delta}_G^!.
\]
Therefore, by (1) we have
\[
\overline{\inprod}(\alpha,\beta)= \pm \phi_P \circ \hat{\Delta}_G^! \circ \hat{p}_G^\natural([S/G]\times [T/G])
= \pm \phi_P \circ d_* ([(S \pitchfork T)/G]).
\]
By naturality we have $\phi_{(S \pitchfork T)/G}= \phi_P \circ d_*$, and hence,
\[
\overline{\inprod}(\alpha,\beta)=\pm \alpha*\beta.
\]
\end{enumerate}
\end{proof}

\section{Homology suspension}\label{sec:suspension}
We investigate the map $\phi_W$ \eqref{eq:phi} as it is interesting in its own right.
For a homotopy commutative square
\begin{equation}\label{dg:phi}
\xymatrix{
W \ar[r]^{\hat{g}} \ar[d]^{\hat{f}} & C \ar[d]^f \\
D \ar[r]^g & A.
}
\end{equation}
a map $\phi_W: H_N(W;R)\to H_{N+1}(A;R)$ when $N>\dim_R(C), \dim_R(D)$ is defined as follows:
First, form the homotopy pushout of the upper-left corner
\begin{equation}\label{dg:B}
\xymatrix{
 W \ar[r]^{\hat{g}} \ar[d]^{\hat{f}} & C \ar[d] \ar[ddr] \\
 D \ar[r] \ar[rrd] & B \ar@{-->}[dr]^q \\ 
 & & A,
 }
\end{equation}
where $q$ is the whisker map.
Look at the Mayer-Vietoris sequence for the homotopy pushout
\[
 H_{N+1}(C;R) \oplus H_{N+1}(D;R) \to H_{N+1}(B;R) \xrightarrow{\partial} 
 H_N(W;R) \to  H_{N}(C;R) \oplus H_{N}(D;R).
\]
By the degree assumption, the both ends are trivial and the boundary operator $\partial$
is an isomorphism.
Then our map is defined by \[\phi_W=q_* \circ \partial^{-1}.\]
Since $\phi_W$ reduces to the ordinary suspension isomorphism when $C=D=pt$ and $A=\Sigma W$, 
we call it the {\em homology suspension}.

We give an alternative description of the homology suspension $\phi_W$.
By replacing $A$ with the mapping cylinder of $f$ and $g$, 
we can assume that $f$ and $g$ are cofibrations with disjoint images.
Let $H: (W\times I, W\times \{0\} \sqcup W\times \{1\}) \to (A, C\sqcup D)$
 be the commuting homotopy, where $I=[0,1]$ is the unit interval.
Since $N>\dim_R(C), \dim_R(D)$, the induced homomorphism of the quotient
$u_*: H_{N+1}(A;R) \to H_{N+1}(A,C\sqcup D;R)$
is an isomorphism. We define $\phi'_{W\to A}$ by the composition
\[
 H_N(W;R) \xrightarrow{\times [(I,\partial I)] } H_{N+1}(W\times I, W\times \partial I;R)
 \xrightarrow{H_*} H_{N+1}(A, C\sqcup D;R) \xrightarrow{u_*^{-1}} H_{N+1}(A;R),
\]
where the first map is the cross product with the generator $[(I,\partial I)]\in H_1(I,\partial I;R)$.

\begin{lem}
The two homomorphism $\phi_W$ and $\phi'_{W\to A}$ coincide.
\end{lem}
\begin{proof}
First, observe that $q_*\circ \phi'_{W\to B}=\phi'_{W\to A}$ for \eqref{dg:B}, and hence, 
it is sufficient to show the statement when \eqref{dg:phi} is a homotopy pushout.
By the naturality of the Mayer-Vietoris sequence,
we have the following commutative diagram, where horizontal sequences are the Mayer-Vietoris sequences:
\[\small
\xymatrix@=10pt{
0=H_N(C\cup (W\times [0,1)))\oplus H_N((W\times (0,1]) \cup D) \ar[r] \ar[d] &
	 H_N(A) \ar[r]^{\partial} \ar[d]^{u_*} & H_{N-1}(W\times (0,1)) \ar[d]^{Id}\ar[r] & 0 \\
0=H_N(C\cup (W\times [0,1)),C)\oplus H_N((W\times (0,1])\cup D,D) \ar[r] & H_N(A,C\sqcup D) \ar[r]^{\partial'} & H_{N-1}(W\times (0,1),\emptyset)
\ar[r] & 0 \\
0=H_N(W\times [0,1),W\times \{0\})\oplus H_N(W\times (0,1], W\times \{1\}) \ar[r] \ar[u] & H_N(W\times I,W\times \partial I)
\ar[r]^{\partial''} \ar[u]^{H_*} & H_{N-1}(W\times (0,1),\emptyset) \ar[u]^{Id} \ar[r] & 0.
}
\]
Notice that $\partial,\partial'$, and $\partial''$ are all isomorphisms.
The assertion follows from the commutative diagram:
\[
\xymatrix{
  H_N(A,C\sqcup D)  \ar[r]^{u_*^{-1}}
	 &  H_N(A)     \\
H_N(W\times I,W\times \partial I)\ar[u]^{H_*} & H_{N-1}(W\times (0,1),\emptyset)\simeq H_{N-1}(W) \ar[l]^(0.6){{\partial''}^{-1}} \ar[u]^{\partial^{-1}},
}
\]
where ${\partial''}^{-1}$ is given by the cross product with $[(I,\partial I)]$.
\end{proof}

With the second definition, we see $\overline{\mu_{M_G}}$ is indeed a ``secondary'' product.
Notice that the induced maps 
$(f\circ \hat{g})_*=(g\circ \hat{f})_*:H_*(W;R)\to H_*(A;R)$
are trivial for $*\ge N$.
The composition of
$\mu_{BG}: H_{k}(S;R) \otimes H_{l}(T;R) \to H_{k+l+\dim(G)}(\pull;R)$ defined for the outer square in \eqref{defining-diagram}
with the induced map of $(f\circ \hat{g})\sim(g\circ \hat{f}): \pull \to (K\times L)_G$
is nothing but $\mu_{BG}: H_{k}(S;R) \otimes H_{l}(T;R) \to H_{k+l+\dim(G)}((K\times L)_G;R)$ defined for the lower-right square in \eqref{defining-diagram},
which is trivial. 
Our secondary product is constructed using the homotopy between $(f\circ \hat{g})$ and $(g\circ \hat{f})$.
\medskip

Finally, we show that the homology suspension $\phi_W$ commutes with the two umkehr maps introduced in \S \ref{sec:primary}.
Given a map $i: \widehat{A}\to A$, consider the pullback of the whole diagram \eqref{dg:phi} via $i$ to have the following cube:
\[
\xymatrix@!0@=45pt{
 & \widehat{W} \ar[rr]\ar'[d][dd] \ar[dl]_i
   & & \widehat{C} \ar[dd] \ar[dl]
\\
 W\ar[rr]\ar[dd]
 & & C \ar[dd]
\\
 & \widehat{D} \ar'[r][rr] \ar[dl]
   & & \widehat{A} \ar[dl]^i
\\
 D \ar[rr]
 & & A
}
\]
\begin{lem}\label{lem:phi-commutes-with-umkehr}
\begin{enumerate}
\item Let $i$ be a fibre bundle with which the Grothendieck bundle transfer is defined (see \S \ref{sec:grothendieck-bundle-transfer}).
Then, we have $\phi_{\widehat{W}}\circ i^\natural= \pm i^\natural\circ \phi_W$
 in the degrees where the both sides are defined.
\item Let $i$ be a (homotopy pullback of a) smooth map with which the Grothendieck transfer is defined (see \S \ref{two-umkehr-maps}).
Then, we have $\phi_{\widehat{W}}\circ i^!= \pm i^!\circ \phi_W$ in the degrees where the both sides are defined.
\end{enumerate}
\end{lem}
\begin{proof}
When $i$ is a fibre bundle, we have a relative bundle map $i:(\widehat{A},\widehat{C}\sqcup \widehat{D}) \to (A, C\sqcup D)$
and the following pullback square
\[
\xymatrix{
(\widehat{W}\times I, \widehat{W}\times \partial I) \ar[d]^i \ar[r]^{\widehat{H}} &   (\widehat{A}, \widehat{C}\sqcup \widehat{D}) \ar[d]^i\\
(W\times I, W\times \partial I) \ar[r]^H & (A, C\sqcup D).
}
\]
Observe that each square in
\[
\xymatrix{
 H_N(W;R) \ar[r]^(0.35){\times [(I,\partial I)] } \ar[d]^{i^\natural} & 
   H_{N+1}(W\times I, W\times \partial I;R) \ar[r]^(0.4){H_*}
  \ar[d]^{i^\natural} & H_{N+1}(A, C\sqcup D;R) \cong H_{N+1}(A;R) \ar[d]^{i^\natural}\\
 H_{N+k}(\widehat{W};R) \ar[r]^(0.35){\times [(I,\partial I)] }  &
  H_{N+k+1}(\widehat{W}\times I, \widehat{W}\times \partial I;R) \ar[r]^(0.4){\widehat{H}_*}
 & H_{N+k+1}(\widehat{A}, \widehat{C}\sqcup \widehat{D};R) \cong H_{N+k+1}(\widehat{A};R)
 }
\]
commutes up to sign by the naturality of the relative version of the Grothendieck bundle transfer (\cite[Chap. V]{B})
with the cross product and pullback.
The first assertion follows from the commutativity of the outer square.
A similar argument proves the second case when $i^!$ is a Grothendieck transfer.
\end{proof}


\end{document}